\documentclass[12pt,letterpaper]{article}

\newcommand{\Keywords}[1]{\par\noindent{\small{\bf Keywords\/}: #1}}
\newcommand{\Class}[1]{\par\noindent{\small{\bf Mathematics Subjects Classification (2010)\/}: #1}}

\usepackage{amsthm, amsmath, amssymb, amsfonts, graphicx, epsfig}
\usepackage{accents}
\usepackage{dsfont}
\usepackage[T1]{fontenc}
	
\usepackage{bbm}
\usepackage{mathrsfs}

\usepackage{ulem}

\usepackage{cancel}

\usepackage{epstopdf}

\usepackage{graphicx}
\usepackage[font=sl,labelfont=bf]{caption}
\usepackage{subcaption}

\usepackage{enumerate}

\usepackage[usenames]{color}
\usepackage{hyperref}

\usepackage{url}
\makeatletter\def\url@leostyle{%
 \@ifundefined{selectfont}{\def\UrlFont{\sf}}{\def\UrlFont{\scriptsize\ttfamily}}} \makeatother

\usepackage[framemethod=default]{mdframed}

\usepackage[margin=1.01in, letterpaper]{geometry}

\newtheorem{theorem}{Theorem}
\newtheorem{proposition}[theorem]{Proposition}
\newtheorem{lemma}[theorem]{Lemma}

\theoremstyle{definition}

\newtheorem{example}[theorem]{Example}
\theoremstyle{remark}
\newtheorem{remark}[theorem]{Remark}

\numberwithin{equation}{section}
\numberwithin{theorem}{section}

\definecolor{Red}{rgb}{0.9,0,0.0}
\definecolor{Blue}{rgb}{0,0.0,1.0}

\def\cB{\mathcal{B}}
\def\cC{\mathcal{C}}

\def\cF{\mathcal{F}}
\def\cG{\mathcal{G}}

\def\cO{\mathcal{O}}
\def\cP{\mathcal{P}}

\def\bE{\mathbb{E}}
\def\bF{\mathbb{F}}
\def\bG{\mathbb{G}}

\def\bP{\mathbb{P}}

\def\bR{\mathbb{R}}

\newcommand{\1}{\mathbbm{1}}            
\newcommand{\set}[1]{\{#1\}}            
\newcommand{\norm}[1]{ \| #1 \| }       
\newcommand{\abs}[1]{\left\vert#1\right\vert}   

\newcommand{\wh}[1]{\widehat{#1}}

 \def\r{{\mathbb R}}

 \def\hat{\widehat}
 
\def\tilde{\widetilde}

 \def\F{{\cal F}}

\newtheorem{thm}{Theorem}[section]
\newtheorem{lem}{Lemma}[section]
\newtheorem{pro}{Proposition}[section]
\newtheorem{cor}{Corollary}[section]
\newtheorem{rem}{Remark}[section]
\newtheorem{ex}{Example}[section]
\newtheorem{defi}{Definition}[section]

  \newcommand{\be}{\begin{equation}}
\newcommand{\ee}{\end{equation}}
\newcommand{\bde}{\begin{displaymath}}
\newcommand{\ede}{\end{displaymath}}
\newcommand{\beq}{\begin{eqnarray*}}
\newcommand{\eeq}{\end{eqnarray*}}
\newcommand{\beqa}{\begin{eqnarray}}
\newcommand{\eeqa}{\end{eqnarray}}
\newcommand{\bel }{\left{\begin{array}{ll}}}
\newcommand{\eel}{\cr \end{array} \right.}
\newcommand{\bd}{\begin{defi}}
\newcommand{\ed}{\end{defi}}
\newcommand{\brem }{\begin{rem} \rm }
\newcommand{\erem }{\end{rem}}
\newcommand{\bex}{\begin{ex} \rm }
\newcommand{\eex}{\end{ex}}
\newcommand{\begth}{\begin{thm}}
\newcommand{\eeth}{\end{thm}}
\newcommand{\bl}{\begin{lem}}
\newcommand{\el}{\end{lem}}
\newcommand{\bp}{\begin{pro}}
\newcommand{\ep}{\end{pro}}
\newcommand{\bcor}{\begin{cor}}
\newcommand{\ecor}{\end{cor}}
\newcommand{\lab }{\label }

 \def\ff{{\mathbb F}}
 
 \def\gg{{\mathbb G}}

\usepackage{fancyhdr}
\usepackage{lastpage}
\title{Semimartingales and Shrinkage of Filtration}

\author{Tomasz R. Bielecki
\\ Department of Applied Mathematics \\
 Illinois Institute of Technology \\
 Chicago, IL 60616, USA \\ \\
Jacek Jakubowski
\\  Institute     of Mathematics  \\ University of Warsaw
\\ Warszawa, Poland \\ \\
Monique Jeanblanc \\
LaMME, Univ Evry, \\
 Universit\'e Paris Saclay, Evry, France \\
 \\
Mariusz Niew\k{e}g\l owski \\
Faculty of Mathematics and Information Science
\\ Warsaw University of Technology
\\ 00-661 Warszawa, Poland}

\date{\vskip 30 pt \today \vskip 25 pt}

\begin{document}
\maketitle
\thispagestyle{empty}	

\begin{abstract}
We consider a complete probability space $(\Omega,\cF,\bP)$, which is endowed with two filtrations, $\bG$ and $\bF$, assumed to satisfy the usual conditions and such that $\bF \subset \bG$. On this probability space we consider a real valued $\bG$-semimartingale $X$. The results can be  generalized to the case of $\bR^n$ valued semimartingales, in a straightforward manner.

The purpose of this work is to study the following two problems:
\begin{itemize}
\item[A.] If $X$ is $\bF$-adapted, compute the $\bF$-semimartingale characteristics of $X$ in terms of the  $\bG$-semimartingale characteristics of $X$.
\item[B.] If $X$ is a special $\bG$-semimartingale but not $\bF$-adapted, compute the $\bF$-semimartingale characteristics of $\bF$-optional projection of $X$ in terms of the $\bG$-canonical decomposition and $\bG$-semimartingale characteristics of $X$.
\end{itemize}
In this paper problem B is solved under the assumption that the filtration $\bF$ is immersed in $\bG$. Beyond the obvious mathematical interest, our study is motivated by important practical applications in areas such as finance and insurance (cf.  \cite{BieJakNie2019}).

\vskip 20 pt
\vskip 20 pt
\Class{60G99,$\,$60H99}
\end{abstract}


\section{Introduction}   \label{sec:intro}

This paper is meant to initiate a systematic study of the change of properties of semimartingales  under shrinkage of filtrations and, when appropriate,  under respective projections. The paper does not aim at a complete and comprehensive study of the topic. Nevertheless, our study contributes, we believe,  to understanding of these problems and to giving, in the some specific cases, explicit solutions.

We consider  a complete probability space $(\Omega,\cF,\bP)$, which is endowed with two filtrations, $\bG$ and $\bF$, assumed to satisfy the usual conditions and such that $\bF \subset \bG$. On this probability space we consider a real valued $\bG$-semimartingale $X$. The results can be  generalized to the case of $\bR^n$ valued semimartingales, in a straightforward manner. We fix a truncation function with respect to which the semimartingale characteristics are computed.

The purpose of this work is to study the following two problems:
\begin{itemize}
	\item[A.] If $X$ is $\bF$-adapted, compute the $\bF$-semimartingale characteristics of $X$ in terms of the  $\bG$-semimartingale characteristics of $X$.
	\item[B.] If $X$ is a special $\bG$-semimartingale but  not $\bF$-adapted,  compute the $\bF$-semimartingale characteristics of the $\bF$-optional projection of $X$ in terms of the $\bG$-canonical decomposition and  $\bG$-semimartingale characteristics of $X$.
\end{itemize}
Note that  $\bG$-semimartingale characteristics  of $X$ are unique up to equivalence, even though they may not determine the law of $X$ uniquely. Thus, the above two problems are well posed.

So, in a sense, we study problems, which are complementary to problems that arise when one studies what happens to a semimartingale under enlargement of filtration, where the main object of interest is study how martingales in a given filtration behave when they are considered in a larger filtration. The goal there is to give, under adequate conditions on the filtration enlargement,  their semimartingale decomposition in this larger filtration. The literature regarding enlargement of filtrations is quite abundant (see, e.g., the recent monograph \cite{AksJea2017} and the references therein). On the contrary, the literature regarding the shrinkage of filtration and its effect on the properties of a semimartingale is essentially non-existent. One can quote the seminal paper of Stricker \cite{Stricker1977}  who establishes that a $\gg$-semimartingale  which is $\ff$-adapted (with $\ff\subset \gg$) is an $\ff$-semimartingale, emphasizing that a $\gg$-local martingale  which is $\ff$-adapted may fail to be an $\ff$-local martingale. But the problem of how the semimartingale characteristics change   under shrinkage of the filtration is not addressed there.   Two notable exceptions are Chapter 4. \S 6 in \cite{LipShi1989} and Section IX.2 in \cite{Jac1979}, that feature partial versions of some of our results. Related study is also done in \cite{BreYor1978} where, however, a different, from our special semimartingales, class of processes was investigated (called {semi-martingales} there). Special cases of our Lemma \ref{lem:heroic-result} are present in the literature in the context of the filtering theory; see for example Lemma 8.4 in \cite{LipShi2001}.
It needs to be stressed that, in general, the problems that we study with regard to shrinkage of filtration are, in general, different from problems studied by the theory of filtering, where,  due to the noise in the observation, the observation filtration is not included in the signal filtration.

Also, contrary to the theory of the enlargement of the filtrations, where only initial and progressive enlargements are studied, here we do not make any specific restrictions regarding relation between the filtrations $\bG$ and $\bF$, except for the inclusion condition  $\bF \subset \bG$, and perhaps some additional conditions, such as the immersion condition in Section \ref{sec:not-adapted}.

An important motivation behind the study originated in this paper is coming from the theory of stochastic structures that has been under works in recent years (cf.  \cite{BieJakNie2019}). One of the problems arising in this theory can be summarized as follows: Suppose that $S=(S^1,\ldots,S^n)$ is a multivariate semimartingale. Suppose that $(B^i, C^i,\nu^i)$ are the semimartingale characteristics of the semimartingale $S^i$ in the natural filtration of $S$. The problem is to find the semimartingale characteristics of $S^i$ in the filtration of a sub-group of coordinates   $S^{i_1},\ldots,S^{i_k}$, $i_1,\ldots,i_k\in \set{1,\ldots,n},$ of $S$, in terms of $(B^i, C^i,\nu^i)$. Once it is understood how to do this, then one can proceed with construction of semimartingale structures, which, by definition, are multivariate semimartingales  whose components are semimartingales with predetermined marginal characteristics in their own filtrations. In a sense, this corresponds to what is being done in the realm of finite dimensional probability distributions via the classical copula theory, where a multivariate distribution is constructed with given margins. This allows for modeling dependence between components of a multivariate random variable, with preservation of the predetermined marginal distributions. Semimartingale structures find applications in areas such as finance and insurance. For example, in insurance, when designing claim policies for a group of claimants it is important to model dependence between multiple claim processes, subject to idiosyncratic statistical properties of these claim processes. In finance, semimartingale structures come in handy for traders who trade basket derivatives, as well as the individual constituents of these derivatives, and need to make sure that respective models for evolution of the basket price process and the price processes of individual constituents are calibrated in a consistent way. In this regard,  quite importantly, semimartingale structures allow for separation of estimating (calibrating) individual (idiosyncratic) characteristics of the components of the structure, from estimation (calibration) of the stochastic dependence between the components of the structure. We refer to \cite{BieJakNie2019} for more applications of stochastic structures.

The paper is organized as follows. In Section \ref{sec:prelim} we formulate the mathematical set-up for our study and we recall some useful concepts and results. In Section \ref{sec:adapted} we study problem A. In Section \ref{sec:not-adapted} we study problem B. In Section \ref{sec:examples} we provide several examples illustrating and complementing our theoretical developments. The complexity of the examples varies. But all of them are meant to illustrate our theoretical developments, even though results presented in some of the examples might possibly be obtained directly.

Finally, in Section \ref{sec:conclusion} we formulate some non-trivial open problems, solution of which will require more in-depth understanding of subject matters discussed in this paper.

\section{Preliminaries}   \label{sec:prelim}

We begin with recalling the concept of characteristics of a semimartingale. These characteristics depend on the choice of filtration and the choice of so called truncation function.  In what follows, we will use the standard truncation function $\chi(x) = x \1_{\abs{x} \leq 1}$.
Given the truncation function $\chi$,  the $\bG$-characteristic triple $(B^\bG, C^\bG, \nu^\bG)$ of a $\bG$-semimartingale $X$ is given in the following way. First
we define the process $X(\chi)$ by
\begin{equation}\label{eq:cutoff}
X_t(\chi) = X_t - X_0 - \sum_{ 0 < s \leq t} (\Delta X_s - \chi(\Delta X_s)), \quad t \geq 0.
\end{equation}
Since
$X(\chi)$ has bounded jumps it is a special $\bG$-semimartingale. Thus, it admits a unique canonical decomposition
\begin{equation}\label{eq:XGspecial}
X(\chi) = B^\bG +  M^{\bG},
\end{equation}
where $M^\bG$ is a $\bG$-local martingale such that  $M^\bG_0 = 0$, and $B^\bG$ is a $\bG$-predictable process with finite variation and $B^\bG_0=0$. The process $B^\bG$ is called the first characteristic of $X$, and this is the only characteristic that depends on the truncation function.

The $\bG$-local martingale $M^\bG$ can be decomposed uniquely  into the sum of two orthogonal martingales $M^\bG = M^{c,\bG} + M^{d, \bG}$, where $M^{c,\bG}$ is a continuous $\bG$-local martingale and $M^{d,\bG}$ is a purely discontinuous $\bG$-local martingale.
It can be shown that $M^{c,\bG}$ does not depend on the choice of  truncation function and it is called the continuous $\bG$-martingale part of $X$ and is denoted by $X^{c,\bG}$. Then, the second characteristic of $X$ is defined as $C^\bG = \langle X^{c,\bG} \rangle$, where $\langle   X^{c,\bG} \rangle$  is the predictable  quadratic variation process of $X^{c,\bG}$. Finally, the third characteristic of $X$ is denoted as $\nu^{\bG}$ and is defined as the $\bG$-predictable measure which is the $\bG$-compensator of $\mu$  -  the jump measure of $X$ as defined in Proposition II.1.16 in \cite{JacShi2002}. Clearly, $\nu^\bG$ does not depend on a choice of truncation function. It is clear that, given a truncation function, the $\bG$-characteristic triple $(B^\bG, C^\bG, \nu^\bG)$ is unique (up to equivalence).

In view of  Proposition II.2.9 in \cite{JacShi2002}, there exists a $\bG$-predictable, locally integrable increasing process, say $A^\bG$, such that
\begin{equation}\label{data}
B^\bG=b^\bG \!\cdot \! A^\bG ,  \qquad C^\bG=c^\bG\!\cdot \! A^\bG ,  \qquad \nu^\bG(dt,dx)=  K^\bG _t(dx) dA^\bG_t,
\end{equation}
where
\begin{itemize}
	\item[i.] $b^\bG$ is an $\bR$-valued and $\bG$-predictable process,
	\item[ii.] $c^\bG$ is an $\bR_+$-valued and $\bG$-predictable process,
	\item[iii.] $K^\bG_t(\omega,dx )$ is a transition kernel from $(\Omega \times \bR_+, {\mathcal P}_\bG)$ to $(\bR,\cB(\bR))$, satisfying condition analogous to condition II.2.11 in \cite{JacShi2002}, and where ${\mathcal P}_\bG$ is the $\bG$-predictable $\sigma$ -field on $\Omega \times \bR_+$,
\end{itemize}
and where $\cdot$ denotes the stochastic or Stieltjes integral, wherever appropriate.

We will assume that
\begin{equation}\label{eq:ass-AG}
A^{\bG}_t=\int_0^t a^{\bG}_udu,
\end{equation}
where $a^{\bG}$ is a $\bG$-progressively measurable process.  This assumption will be satisfied in examples studied in Section \ref{sec:examples}.

In what follows we use the following notions and notation:

\begin{enumerate}
	\item  For a given process $Z$, we denote by $^{o,\bF}Z$ the optional projection of $Z$ on $\bF$ defined in the sense of He et al. \cite{HeWanYan1992}, i.e.,
	the unique $\bF$-optional, finite valued  process such that for every $\bF$-stopping time $\tau$  we have
	\[
	\bE ( Z_\tau \1_{\tau < \infty} | \cF_\tau) =\, ^{o,\bF}\!Z_\tau \1_{\tau < \infty}.
	\]
	Note that by Theorem 5.1 in \cite{HeWanYan1992} this optional projection exists if $Z$ is a measurable process such that $Z_\tau\1_{\tau < \infty}$ is  $\sigma$-integrable with respect to $\cF_\tau $ for every $\bF$-stopping time $\tau$. That is, there exists  a sequence  of sets $(A_n)_{n=1}^\infty$ such that $ A_n  \in \cF_\tau $, $A_n \uparrow \Omega$ and $\bE(Z_\tau\1_{\tau < \infty}\1_{A_n})<\infty$ for $n=1,2,\ldots $.
	
	\item  For a given process $Z$, we denote by $^{p,\bF}Z$ the predictable projection of $Z$ on $\bF$ defined in the sense of He et al. \cite{HeWanYan1992}, i.e.,
	the unique $\bF$-predictable, finite valued  process such that for every $\bF$-predictable stopping time $\tau$  we have
	\[
	\bE ( Z_\tau \1_{\tau < \infty} | \cF_{\tau-}) =\, {^{p,\bF}\!Z}_\tau \1_{\tau < \infty}.
	\]
	Note that by Theorem 5.2 in \cite{HeWanYan1992} this predictable projection exists if $Z$ is a measurable process such that $Z_\tau\1_{\tau < \infty}$ is  $\sigma$-integrable with respect to $\cF_{\tau-} $ for every predictable $\bF$-stopping time $\tau$. That is, there exists  a sequence of sets $(A_n)_{n=1}^\infty$ such that $ A_n  \in \cF_{\tau-} $, $A_n \uparrow \Omega$ and $\bE(Z_\tau\1_{\tau < \infty}\1_{A_n})<\infty$ for $n=1,2,\ldots $.
	
	\item We will also need a notion  of $\bF$-optional and $\bF$-predictable projections for any function $W: \tilde{\Omega} \rightarrow \bR$, which is  measurable with respect to $\tilde{\cF}$, 	where
	\[	
	\tilde{\Omega} := \Omega \times \bR_+ \times \bR, \quad
	\tilde{\cF} := \cF \otimes \cB(\bR_+) \otimes \cB(\bR)\,.
	\]
	The $\bF$-optional projection  of such a function $W$ is defined as the jointly measurable function  $^{o,\bF}W$  on $\Omega \times \bR_+ \times  \bR$,
	which is such that for all $x \in  \bR$ the process $^{o,\bF}W(\cdot,x)$ is the optional projection on $\bF$ of the process $W(\cdot,x)$. Similarly, the $\bF$-predictable projection  of such a function $W$ is defined as the function $^{p, \bF}W$  on $\Omega \times \bR_+ \times  \bR$, which is such that for all $x \in  \bR$ the process $^{p,\bF}W(\cdot,x)$ is the predictable projection on $\bF$ of the process $W(\cdot,x)$.
	
	\item We denote by $\cO_\bF$ (resp. $\cP_\bF$), the $\bF$-optional (resp. the $\bF$-predictable) sigma-field on $\Omega \times \bR_+$ generated by $\bF$-adapted c\`adl\`ag (resp. continuous)  processes. Analogously we introduce the sigma fields $\tilde{\cO}_\bF$ and $\tilde{\cP}_\bF$  on  $\tilde{\Omega}$
	defined by
	\[
	\tilde{\cO}_\bF : = \cO_\bF \otimes \cB(\bR), \qquad
	\tilde{\cP}_\bF : = \cP_\bF \otimes \cB(\bR).
	\]
	
	\item A random measure $\pi$ on $\cB(\bR_+) \otimes \cB(\bR)$ is $\bF$-optional (resp. $\bF$-predictable)  if  for any $\tilde{\cO}_\bF $-measurable (resp. $\tilde{\cP}_\bF$-measurable)  positive real function $W$, the real valued process
	\[
	V (\omega, t) := \int_{[0,t] \times \bR } W(\omega, s, x ) \pi(\omega; ds, dx )
	\] is $\bF$-optional (resp. $\bF$-predictable);
	equivalently if for any positive real, measurable function $W$  on $\widetilde \Omega$
	\[
	\bE \bigg(\int_{\bR_+ \times \bR }W( s, x ) \pi(ds, dx )\bigg)
	=
	\bE \bigg(\int_{\bR_+ \times \bR } {^{q,\bF}}W( s, x ) \pi(ds, dx )\bigg),
	\]
	where $q = o$ (resp. $q=p$).

	\item We say that a random measure $\pi$ on $\cB(\bR_+) \otimes \cB(\bR)$ is $\bF$-optionally (resp. $\bF$-predictably), $\sigma$-integrable if the measure $M_\pi $ on $\tilde{\cF}$ defined by
	\begin{equation}\label{mpi}
	M_\pi (\tilde{B}) := \bE \bigg( \int_{\bR_+ \times \bR } \1_{\tilde{B} }(\omega, t,x) \pi(\omega, dt, dx) \bigg), \qquad \tilde{B} \in \tilde{\cF},
	\end{equation}
	restricted to $\tilde{\cO}_\bF$ (resp. to $\tilde{\cP}_\bF$),  is a $\sigma$-finite measure. In other words $\pi$ is $\bF$-optionally, resp. $\bF$-predictable, $\sigma$-integrable if there exist a sequence of sets {$ (\tilde{A}_k)_{k=1}^\infty$  such that $\tilde{A}_k \in \tilde{\cO}_\bF$ (resp. $ \tilde{A}_k \in \tilde{\cP}_\bF$),     with $M_\pi(\tilde{A}_k) < \infty$ for each $k$ and $\tilde{A}_k \uparrow \tilde{\Omega}$.}

	\item  For a random measure $\pi$ on $\cB(\bR_+) \otimes \cB(\bR)$ we denote by $\pi^{o,\bF}$ the $\bF$-dual optional projection of $\pi$ on $\bF$, i.e., the unique $\bF$-optional  measure on $\cB(\bR_+) \otimes \cB(\bR)$ such that   it is $\bF$-optionally $\sigma$-integrable  and for every positive $\tilde{\cO}_\bF$-measurable function $W$ on $\widetilde \Omega$, we have
	\[
	\bE \Big( \int_{\bR_+ \times \bR } W (t,x) \pi(dt,dx) \Big) = \bE \Big( \int_{\bR_+ \times \bR } W (t,x) \pi^{o,\bF} (dt,dx) \Big).
	\]
	The $\bF$-dual predictable projection of $\pi$ on $\bF$, denoted by $\pi^{p,\bF}$, is defined analogously,  as the unique $\bF$-predictable  measure on $\cB(\bR_+) \otimes \cB(\bR)$ such that
	it is $\bF$-predictably $\sigma$-integrable  and for every positive $\tilde{\cP}_\bF$-measurable function $W$ on $\widetilde \Omega$, we have
	\[
	\bE \Big( \int_{\bR_+ \times \bR } W (t,x) \pi(dt,dx) \Big) = \bE \Big( \int_{\bR_+ \times \bR } W (t,x) \pi^{p,\bF} (dt,dx) \Big).
	\]
	We note that existence  and uniqueness of $\pi^{o,\bF}$ (resp. $\pi^{p,\bF} (dt,dx)$) holds under assumption that $\pi$ is $\bF$-optionally (resp. $\bF$-predictably), $\sigma$-integrable (see e.g. \cite[Theorem 11.8]{HeWanYan1992}).
	
	\item For any process $A$ and any (stopping) time $\vartheta$, we denote by $A^\vartheta$ the process $A$ stopped at $\vartheta$.
	
	\item We use the standard notation $[\cdot,\cdot]$ (resp. $[\cdot]$) for the quadratic co-variation (resp. variation) of  real-valued semimartingales. We denote by $\langle\cdot,\cdot \rangle $ (resp. $\langle \cdot\rangle $) the $\mathbb{F}$-predictable quadratic co-variation (resp. variation) of  real-valued semimartingales. We denote by  $\langle\cdot,\cdot \rangle^{\mathbb{G}} $ (resp. $\langle \cdot\rangle^{\mathbb{G}} $) the $\mathbb{G}$-predictable quadratic co-variation (resp. variation) of real-valued semimartingales.
	
	\item We use the usual convention that $\int_0^t = \int_{(0,t]}$, for any $t\geq 0$.
	
\end{enumerate}

In the rest of the paper we shall study the $\bF$-characteristics of $X$ in the case when $X$ is $\bF$-adapted, and the $\bF$-characteristics of the optional projection of $X$ on $\bF$ in the case when $X$ is not $\bF$-adapted, providing conditions for such optional projection to exist.

\section{Study of Problem A: The Case of $X$ adapted to $\bF$}   \label{sec:adapted}

In this section, we consider the case where  $X$  is a $\bG$-semimartingale, which is $\bF$-adapted. Thus, it is an $\bF$-semimartingale (see \cite[Theorem  3.1]{Stricker1977}).

We start with a result regarding  semimartingales with deterministic $\bG$-characteristics, that is, semimartingales with independent increments (cf. \cite[Theorem II.4.15]{JacShi2002}).
\begin{proposition}\label{prop:det-char}
	If $X$ has deterministic $\bG$-characteristics, then they are also $\bF$-characteristics of $X$.
\end{proposition}
\begin{proof}
	
	In the proof we will use the standard notation for integral of $g$ with respect to a measure $\gamma$, that is
	\[
	(g*\gamma)_t : = \int_0^t \int_\r g(x)  \gamma(ds,dx ), \quad t \geq 0.\]
	Note that $X(\chi)$ given in \eqref{eq:cutoff} is an $\bF$-adapted process with jumps bounded by $1$.
	From assumption \eqref{eq:ass-AG} it follows that the process $B^\bG $ does not have jumps. Given this, from \cite[Theorem II.2.21]{JacShi2002} we know that the processes
	\begin{equation}\label{eq:locmart-char}
	M^\bG = X(\chi) - B^\bG ; \quad Y = (M^\bG)^2 - C^\bG - \chi^2(x) * \nu^\bG ;
	\quad
	g*\mu - g * \nu^{\bG}, \quad g \in \cC^+(\bR),
	\end{equation}
	are $\bG$-local martingales, where the class $\cC^+(\bR)$ of functions is defined in \cite[II.2.20]{JacShi2002}. These processes are $\bF$-adapted processes since the $\bG$ characteristics of $X$ are assumed to be deterministic.
	Thus, by  \cite[Theorem 3.1]{Stricker1977}, $M^\bG$ is an $\bF$-semimartingale.
	Moreover, since $M^\bG$ has bounded jumps, it is a special $\bF$-semimartingale. So, by  \cite[Theorem 2.6]{Stricker1977}  it is an $\bF$-local martingale.
	In order to analyze the $\bG$--local martingale $Y$ in \eqref{eq:locmart-char} let us take a localizing sequence of $\bG$-stopping times $(\tau_n)_{n=1}^\infty$. Then we have
	\[
	\bE	(M^\bG)^2_{\tau_n \wedge \sigma} = \bE( C^\bG + \chi^2 * \nu^\bG)_{\tau_n \wedge \sigma}, \quad n=1,2,\ldots ,
	\]
	for every bounded  $\bG$-stopping time $\sigma$. This implies that
	\[
	\bE( C^\bG + \chi^2 * \nu^\bG)_{\tau_n \wedge \sigma}  \leq \bE( C^\bG + \chi^2* \nu^\bG)_{T}
	=
	C^\bG_T + (\chi^2 * \nu^\bG)_{T} < \infty,
	\]
	since the  $\bG$-characteristics of $X$ are deterministic. So, upon letting $n \rightarrow \infty$, we obtain
	\[
	\bE\Big(	(M^\bG)^2_{\sigma} - C^\bG_\sigma - (\chi^2 * \nu^\bG)_{\sigma} \Big) = 0.
	\]
	Since $\bF \subset \bG$ the above holds for all bounded $\bF$-stopping times. Thus $(M^\bG)^2 - C^\bG - \chi^2* \nu^\bG$ is an $\bF$-martingale.
	
	Similar reasoning can be used to show that the third $\bG$--local martingale in \ref{eq:locmart-char} is
	an $\bF$-martingale for every $g \in \cC^+(\bR)$. Indeed, let is fix $g \in \cC^+(\bR)$.
	Then there exist a localizing sequence of $\bG$-stopping times $(\tau_n)_{n=1}^\infty$ such that
	\[
	\bE (g*\mu^X)_{\sigma \wedge \tau_n} = \bE(g * \nu^{\bG})_{\sigma \wedge \tau_n}, \quad n=1,2,\ldots ,	
	\]
	for an arbitrary bounded $\bG$-stopping time $\sigma$. By letting $n \rightarrow \infty$ we obtain
	\[
	\bE (g*\mu)_{\sigma } = \bE(g * \nu^{\bG})_{\sigma } \leq (g * \nu^{\bG})_{T} < \infty.
	\]
	This implies that $g*\mu^X - g * \nu^{\bG}$ is an $\bF$-martingale.	
	Consequently all $\bG$--local martingales defined in \eqref{eq:locmart-char} are $\bF$--local martingales.
	Using again  \cite[Theorem II.2.21]{JacShi2002} we finish the proof.
	
\end{proof}

\noindent
We now proceed to  consider a general case of $\bG$-characteristics of $X$. Towards this end, we make the following assumptions:

A1. For every $t \geq 0$  we have
\[
\bE  \big( \int_0^t | b^\bG_u | a^\bG_u  du \big)  < \infty,
\]
where $a^\bG\geq 0$ is defined in \ref{eq:ass-AG}.

A2. The process $b^\bG a^\bG$ admits an $\bF$-optional projection.

A3. The process $M^\bG$ defined in \eqref{eq:XGspecial} is a true $\bG$-martingale.

\begin{remark}	
	In view of assumption A3, {the} process $^{o,\bF} M^\bG$ is a true $\bF$-martingale as well. If the process $M^\bG$ were a  $\bG$-local-martingale but not a true  $\bG$-martingale, then $^{o,\bF} M^\bG$  might not necessarily be an  $\bF$-local-martingale. See \cite[section 2]{Stricker1977} and  \cite{FolProt2011}.
\end{remark}

\medskip
\noindent
We will need the following two technical results.
\begin{lemma}\label{lem:pre-heroic-result}
	Suppose that $A$  is an $\bF$-adapted process with prelocally integrable variation, $H$ is a process admitting an $\bF$-optional projection, and such that $H\cdot A$ has an integrable variation. Then
	\[
	M = ^{o, \bF}\!\!(H\!\cdot\!A) - (^{o, \bF}\!H)\!\cdot\!A
	\]
	is a uniformly integrable $\bF$-martingale.
\end{lemma}
\begin{proof}
	Applying  \cite[Corollary 5.31.(2)]{HeWanYan1992} to the  process $H\cdot A$,  we conclude that the process
	\[
	M = ^{o, \bF}\!\!(H\!\cdot\!A) - (H\!\cdot\!A)^{o, \bF}
	\]
	is a uniformly integrable martingale. 	Now, by  \cite[Theorem 5.25]{HeWanYan1992} and the remark following this theorem, we have
	\[
	(H \cdot A )^{o, \bF} = (^{o, \bF}\!H) \cdot A \, ,
	\]
	which finishes the proof.
\end{proof}			

\begin{lemma}\label{lem:heroic-result}
	Let A1 and A2 be satisfied. Then, $^{o,\bF} B^\bG$ and $\int_0^\cdot \,  {^{o,\bF}} (b^\bG a^\bG)_udu$ exist
	and
	the process $M^B$ given as
	\begin{equation}\label{eq:martBG}
	M^B_t={^{o,\bF} B^\bG_t}- \int_0^t\,  {^{o,\bF}} (b^\bG a^\bG)_udu, \qquad t \geq  0,
	\end{equation}
	is an  $\bF$-martingale. Moreover, if  $\bF$ is $\bP$-immersed in $\bG$,\footnote{ We recall that $\bF$ is $\bP$-immersed in $\bG$ if any $(\bF,\bP)$-martingale is a $(\bG,\bP)$-martingale.} then $M^B$ is a null process.
\end{lemma}
\begin{proof}
	Since, by assumption A1, the process $B^\bG_t = \int_0^t H_s ds$, where $	H_s = a^\bG_s b^\bG_s $, is prelocally integrable and, by assumption A2, $H$  has an optional projection, we may apply  \cite[Theorem 5.25]{HeWanYan1992} and conclude
	that  $^{o,\bF} B^\bG$ and $\int_0^\cdot \,  {^{o,\bF}} (b^\bG a^\bG)_udu$ exist.
	
	Now, fix $T>0$ and let
	\[
	L_t = a^\bG_t b^\bG_t \1_{\set{ t \leq T}}, \quad t\geq 0.
	\]
	Then, invoking A1, A2 and applying Lemma \ref{lem:pre-heroic-result} with $A_t=t$ we conclude that
	$^{o,\bF} ( L\!\cdot\! A)$ and $ {^{o,\bF}}L\!\cdot\! A $ exist  and the process
	\begin{equation}\label{eq:martofL.A}
	N_t :=	\,^{o,\bF}\!( L\!\cdot\!A)_t - (
	{^{o,\bF}}L\!\cdot\!A )_t, \quad t \geq 0,
	\end{equation}
	is a uniformly integrable  $\bF$-martingale.
	Now note that $L\!\cdot\!A = (L\!\cdot\!A)^T = (H\!\cdot\!A)^T$ so, by \cite[Theorem 5.7]{HeWanYan1992}, we have for $t\in [0,T]$
	\begin{equation}\label{eq:ofL.A1}
	^{o,\bF}\! ( L\!\cdot\!A)_t =\,
	^{o,\bF}\! ( (H\!\cdot\!A)^T )_t =\, ^{o,\bF}\! ( H\!\cdot\!A )_{t}.
	\end{equation}
	Using the definition of $L$ and applying again \cite[Theorem 5.7]{HeWanYan1992}  we have
	\[
	{^{o,\bF}}L = {^{o,\bF}}(L 1_{[0,T]}) = 1_{[0,T]}
	{^{o,\bF}}L =  1_{[0,T]} {^{o,\bF}}(L^T) =1_{[0,T]} {^{o,\bF}}(H^T) = 1_{[0,T]} {^{o,\bF}}H,
	\]
	which implies that
	\begin{equation}\label{eq:ofL.A2}
	{^{o,\bF}}L\!\cdot\!A =  (1_{[0,T]}{^{o,\bF}}H )\!\cdot\!A = ({^{o,\bF}}H \!\cdot\!A)^T.
	\end{equation}
	Using \eqref{eq:ofL.A1} and \eqref{eq:ofL.A2} we see that the martingale $N$ defined by \eqref{eq:martofL.A} can be written on $[0,T]$ as
	\[
	N_t = {^{o,\bF}\!( H\!\cdot\!A )_{t}} -  ({^{o,\bF}}H\!\cdot\!A)_t={^{o,\bF}\! B^\bG_t}- \int_0^t\,  {^{o,\bF}} (b^\bG a^\bG)_udu, \qquad t \in [0,T].
	\]
	Since $T$ was arbitrary, this proves that the  process given  by \eqref{eq:martBG} is an $\bF$-martingale.
	
	Finally, we will now prove that if  $\bF$ is $\bP$-immersed in $\bG$, then the martingale $M^B$ is a null process. Indeed, for any $t\geq 0$, we have
	\begin{align*}
	{^{o,\bF} B^\bG_t}&=\bE\left( \int_0^t  b^\bG_u a^\bG_u  du \vert \cF_t \right )= \int_0^t \bE\left( b^\bG_u a^\bG_u \vert \cF_t \right ) du\\
	&=\int_0^t \bE\left( b^\bG_u a^\bG_u \vert \cF_u \right )du = \int_0^t\,  {^{o,\bF}} (b^\bG a^\bG)_udu,
	\end{align*}
	where the third equality is a consequence  of immersion of $\bF$  in $\bG$ .
	
	The proof of the lemma is complete.
\end{proof}

\begin{remark}\label{lem:heroic-remark}
	It is important to note that Lemma \ref{lem:heroic-result} is true regardless whether the process $X$ is adapted with respect to  $\bF$ or not. Special versions of this lemma are known in the filtering theory. See for example Lemma 8.4 in \cite{LipShi2001}, or the proof of Theorem 8.11 in \cite{RW2000}.
\end{remark}

The next theorem is the main result in this section.

\begin{theorem}\label{adapted} Assume A1-A3. Then, the $\bF$-characteristic triple of $X$ is given as
	\[
	B^\bF =  \int_0^\cdot {^{o,\bF}}(b^{\bG} a^{\bG})_sds, \quad
	C^\bF = C^\bG, \quad
	\nu^\bF(dt,dx) = \left ( K^\bG_t(dx)  a^\bG_tdt \right )^{p, \bF}.
	\]
\end{theorem}

\begin{proof}
	Let us consider the process $X(\chi)$ given by \eqref{eq:cutoff}.  As observed earlier,  $X(\chi)$ is a $\bG$-special semimartingale with unique canonical decomposition \eqref{eq:XGspecial}. Since $X(\chi)$ is $\bF$-adapted, it is also an $\bF$-special semimartingale, with unique canonical decomposition, say
	\begin{equation}\label{eq:Fspecial}
	X(\chi) = B^\bF + M^\bF,
	\end{equation}
	where $B^\bF$ is $\bF$-predictable process of finite  variation and $M^\bF$ is an $\bF$-local martingale. The process $B^\bF$ is the first characteristic in the $\bF$-characteristic triple of $X$.
	
	Our first goal is to provide a formula for $B^\bF$ in terms of the $\bG$-characteristics of $X$. Towards this end we first observe that from Lemma \ref{lem:heroic-result} if follows that $^{o,\bF} B^\bG$ exists. Recall that by assumption A3 the process $M^\bG$ showing in \eqref{eq:XGspecial} is a $\bG$-martingale.  Since for  any bounded $\bF$-stopping time $\tau \leq T $, using the fact that $\tau$ is a $\bG$-stopping time and Doob's optional stopping theorem, we have
	\[
	\bE M^\bG_\tau = \bE M^\bG_0 < \infty .
	\]
	Thus $M^\bG$ is $\sigma$-integrable with respect to $\cF_\tau$ for every bounded $\bF$-stopping time $\tau$, so
	its optional projection $^{o,\bF} M^\bG$ exists   (see
	\cite[Theorem 5.1]{HeWanYan1992}). From this, from \eqref{eq:XGspecial} and from the linearity of the optional projection  we  conclude that the optional projection ${^{o,\bF}}X(\chi)$ exists, and is given as
	\begin{equation}\label{eq:op1}
	{^{o,\bF}}X(\chi)= \,^{o,\bF} M^\bG+\,^{o,\bF} B^\bG.
	\end{equation}
	Since $X(\chi)$ is $\ff$-adapted we have
	\begin{equation}\label{eq:op2}
	{^{o,\bF}}X(\chi) = X(\chi) .
	\end{equation}
	Combining \eqref{eq:op1} and \eqref{eq:op2} we obtain
	\[
	X(\chi)= \,^{o,\bF} M^\bG+\,^{o,\bF} B^\bG.
	\]
	Thus, since  the process  $^{o,\bF} B^\bG_t- \int_0^t\,  {^{o,\bF}} (b^\bG a^\bG)_udu$ is an  $\bF$-martingale (by Lemma \ref{lem:heroic-result} again),  and  since, in view of A3, the process $^{o,\bF} M^\bG$ is an $\bF$-martingale, we see that
	\[
	X_t(\chi) =M^\bF_t+\int _0 ^ t  {^{o,\bF}} (b^\bG a^\bG)_s  ds,
	\]
	where $M^\bF_t= \,^{o,\bF} M^\bG_t+\, ^{o,\bF} B^\bG_t- \int_0^t\,  {^{o,\bF}} (b^\bG a^\bG)_udu$. Thus, by uniqueness of the decomposition \eqref{eq:Fspecial} of the special $\bF$-semimartingale $X$, we conclude that
	\[
	B^\bF_t=  \int _0 ^ t {^{o,\bF}}(b^\bG a^\bG)_sds.
	\]
	\noindent
	The second formula, $C^\bF = C^\bG$, follows from \cite[Remark 9.20]{Jac1979}.
	
	\medskip
	It remains to derive a formula for $\nu^\bF$. Towards this end, we recall that ${\nu^\bG} = \mu^{p,\bG}$ is a $\tilde{\cP}_\bG$-predictably   $\sigma$-integrable random measure (i.e., using the notation \eqref{mpi}, $M_{\nu^\bG}$ is $\sigma$-finite on $\tilde{\cP}_\bG$,   such that
	\[
	M_{\nu^\bG}|_{\tilde{\cP}_\bG} =  M_{\mu}|_{\tilde{\cP}_\bG}.
	\]
	Thus since  $\tilde{\cP}_\bF \subset \tilde{\cP}_\bG$ we have
	\begin{equation}\label{eq:Mnu=Mmu}
	M_{\nu^\bG}|_{\tilde{\cP}_\bF} =  M_{\mu}|_{\tilde{\cP}_\bF}.
	\end{equation}
	Since $M_{\mu}$ is $\sigma$-finite on ${\tilde{\cP}_\bF}$ (see the proof of Theorem 11.15 \cite{HeWanYan1992}, with ${\tilde{\cP}}$ there replaced by ${\tilde{\cP}_\bF}$),
	the above implies that $M_{\nu^\bG}$ is also  $\sigma$-finite on ${\tilde{\cP}_\bF}$. So $\nu^\bG$ is $\bF$-predictably $\sigma$-integrable. Thus it has the $\bF$-dual predictable projection $(\nu^\bG)^{p,\bF}$ which is characterized by
	\begin{equation}\label{eq:MnuGpF=MnuG}
	M_{(\nu^\bG)^{p, \bF }}|_{\tilde{\cP}_\bF} =  M_{\nu^\bG}|_{\tilde{\cP}_\bF}.
	\end{equation}
	This and \eqref{eq:Mnu=Mmu} imply that
	\[
	M_{(\nu^\bG)^{p, \bF }}|_{\tilde{\cP}_\bF} = M_{\mu}|_{\tilde{\cP}_\bF}.
	\]
	So, by the uniqueness of dual predictable projections we have $(\nu^\bG)^{p, \bF } = \nu^\bF$.
	The proof is complete.
\end{proof}

\begin{remark}
	Let us note that we also have
	\[
	(\nu^\bG)^{p,\bF} = ((\nu^\bG)^{o,\bF})^{p,\bF}.
	\]
	Indeed, by analogous reasoning as in the proof of  \cite[Theorem 11.8]{HeWanYan1992} we can prove that the random measure $\nu^\bG$ admits an $\bF$-dual optional projection if and only if
	it is $\bF$-optionally $\sigma$-integrable. Now, recall that $M_{\nu^\bG}$ is  $\sigma$-finite on ${\tilde{\cP}_\bF}$.
	This and the fact that ${\tilde{\cP}_\bF} \subset {\tilde{\cO}_\bF}$ imply that  $M_{\nu^\bG}$ is also  $\sigma$-finite on ${\tilde{\cO}_\bF}$, so $\nu^\bG$ is  $\bF$-optionally $\sigma$-integrable. 	Thus there exists $ (\nu^\bG)^{o,\bF}$ -- the $\bF$-dual optional projection of $\nu^\bG$, i.e., the unique  $\bF$-optional measure  which  is $\bF$-optionally $\sigma$-integrable such that
	\[
	M_{\nu^\bG}|_{\tilde{\cO}_\bF} =  M_{(\nu^\bG)^{o,\bF}}|_{\tilde{\cO}_\bF}.
	\]
	Hence we  have
	\begin{equation}\label{eq:Mnu=Mnuo}
	M_{\nu^\bG}|_{\tilde{\cP}_\bF} =  M_{(\nu^\bG)^{o,\bF}}|_{\tilde{\cP}_\bF}\,.
	\end{equation}
	Since $M_{\nu^\bG}$ is $\sigma$-finite on $\tilde{\cP}_\bF$, so is $M_{(\nu^\bG)^{o,\bF}}$. Therefore, invoking again  \cite[Theorem 11.8]{HeWanYan1992}, we conclude that there exists the $\bF$-predictable projection of $(\nu^\bG)^{o,\bF}$, i.e. $((\nu^\bG)^{o,\bF})^{p,\bF}$,  for which we have
	\[
	M_{((\nu^\bG)^{o,\bF})^{p,\bF}}|_{\tilde{\cP}_\bF} =  M_{(\nu^\bG)^{o,\bF}}|_{\tilde{\cP}_\bF}.
	\]
	From the latter equality and from \eqref{eq:Mnu=Mmu} and \eqref{eq:Mnu=Mnuo} we deduce that
	\[
	M_{((\nu^\bG)^{o,\bF})^{p,\bF}}|_{\tilde{\cP}_\bF}  = M_{\mu}|_{\tilde{\cP}_\bF}.
	\]
	By uniqueness of the $\bF$-dual predictable projection of $\mu$ we finally obtain  \[\nu^\bF = \mu^{p, \bF} = ((\nu^\bG)^{o,\bF})^{p,\bF} .\] 		\qed
	
\end{remark}

\bigskip

\noindent {\bf The case of immersion between $\bF$ and $\bG$.}

\medskip

We briefly discuss here the case when $\bF$ is $\bP$-immersed in $\bG$.
We will show that 	$(B^\bF ,C^\bF,\nu^\bF) = (B^\bG ,C^\bG,\nu^\bG)$. Towards this end, let us consider the process $X(\chi)$ defined by \eqref{eq:XGspecial}.

Clearly, the process $X(\chi)$ has bounded jumps and is both $\bG$-adapted  and $\bF$-adapted.
Thus, it is a special semimartingale in both filtrations, and hence it has the canonical decompositions
\[
X(\chi)  = M^\bF+B^\bF =  M^\bG + B^\bG.
\]
Since, by immersion, $M^\bF$ is a $\bG$-martingale and, obviously, $B^\bF$  is $\bG$-predictable, one has that  $M^\gg=M^\ff$ and $B^\gg=B^\ff$ (by uniqueness of canonical $\bG$-decomposition of $\breve{X}$).

\medskip
\noindent  The fact that $C^\bF = C^\bG$ follows, again, from \cite[Remark 9.20, p.288]{Jac1979}.

\medskip
\noindent
Finally, we verify that $\nu^\bG = \nu^\bF$. Note that, for any positive real measurable function $g$, the process
$g * \mu - g * \nu^\bF$ is an $\bF$-local martingale and hence, by immersion, a $\bG$-local martingale. This implies, by uniqueness of the compensator and by the fact that $\nu^\bF$ is $\bG$-predictable, that $\nu^\bF = \nu^\bG$.

In conclusion, we have

\begin{proposition} Assume that $\bF$ is $\bP$-immersed in $\bG$. Then,
	\[(B^\bF ,C^\bF,\nu^\bF) = (B^\bG ,C^\bG,\nu^\bG). \]
\end{proposition}

\section{Study of Problem B: The Case of $X$ not adapted to $\bF$}   \label{sec:not-adapted}

In this section we consider the case where  $X$ is a $\gg$-special semimartingale, but it is not adapted to $\bF$. Therefore, we shall study here the $\bF$-optional projection $^{o,\bF} X$ of $X$ and its semimartingale characteristics. In particular, in Theorem \ref{not-adapted} we provide sufficient conditions on $X$ under which the $\bF$-optional projection $^{o,\bF} X$ exists and is an $\bF$-special semimartingale.

We have the following canonical decompositions of $X$ (see \cite[Corollary 11.26]{HeWanYan1992} or  \cite[Corollary II 2.38]{JacShi2002}):
\begin{equation}\label{canG}
X_t =X_0+X^{c,\bG}_t+\int _0^t \int_\bR  x  (\mu^\bG (ds,dx)-\nu^\bG (ds,dx))+\wh{B}_t^\bG = X_0 + \wh{M}^\bG_t+\wh{B}_t^\bG,
\end{equation}
where $\wh{B}^\bG$,  called the modified first characteristic,  and  the $\gg$ local martingale $\wh{M}^\bG$  are given by
\begin{align}\label{eq:modBG}
\wh{B}^\bG_t &= {B}^\bG_t + \int_0^t \int_{|x| >1}  x \nu^\bG(ds,dx), \qquad \qquad t \geq 0,
\\ \nonumber
\wh{M}^\bG_t &= X^{c,\bG}_t+\int _0^t \int_\bR  x  (\mu^\bG (ds,dx)-\nu^\bG (ds,dx)) \\ \nonumber
&= M^\bG_t + \int_0^t \int_{|x| >1}  x (\mu^\bG(ds,dx) - \nu^\bG(ds,dx)), \quad t \geq 0.
\end{align}
Moreover if $^{o,\bF} X$ exists and is an $\bF$-special semimartingale, then
\begin{equation}\label{canF}
^{o,\bF} X_t=X_0+(^{o,\bF} X)^{c,\bF}_t+\int _0^t \int_\bR x (\mu^\bF (ds,dx)-\nu ^\bF(ds,dx))+B_t^\bF = X_0 + \wh{M}^\bF_t+\wh{B}_t^\bF .
\end{equation}
Note that  $\mu^\bF$ is the random measure of jumps of $^{o,\bF}X$ and $\nu^\bF$ denotes its $\bF$-compensator whereas in \eqref{canG} $\mu^\bG$ denotes random measure of jumps of $X$ and $\nu^\bG$ its $\bG$-compensator.
We also  note that  the first $\bF$-characteristic  $B^\bF$ can be computed from the modified one (i.e. $\wh{B}^\bF$) and $\nu^\bF$ by means of a counterpart of the formula \eqref{eq:modBG}
i.e.
\begin{equation}\label{eq:modBF}
\wh{B}^\bF_t = {B}^\bF_t + \int_0^t \int_{|x| >1}  x \nu^\bF(ds,dx).
\end{equation}

In this section we will work under the following additional standing assumptions:

\medskip

{$\hat{\textrm{A1}}$.} For every $t \geq 0$  we have
\[
\bE  \big( \int_0^t | \hat b^\bG_u | a^\bG_u  du \big)  < \infty,
\]
where $a^\bG\geq 0$ is defined in \ref{eq:ass-AG}, and

\begin{equation}\label{eq:def:hatbG}
\wh{b}^\bG_t = {b}^\bG_t +
\int_{|x| > 1} x K^\bG_t(dx), \quad t \geq 0.
\end{equation}

{$\hat{\textrm{A2}}$.} The process $\wh b^\bG a^\bG$ admits an $\bF$-optional projection.

{$\hat{\textrm{A3}}$.} $\wh{M}^\bG$ is a  square integrable martingale.\footnote{Note that in Assumption A3 we postulated that $M^\bG$ is a martingale, but not necessarily a  square integrable martingale.}

{B1.} There exists a square integrable $\bF$-martingale $Z$ such that the predictable representation property holds for  $(\bF,Z)$: any square integrable $\bF$-martingale $M$ admits {a} representation $M_t=M_0+\int_0^t \psi_u\, dZ_u$, $t\geq 0$, with an $\bF$-predictable process $\psi$.

{B2. }The $\bF$-martingale $Z$ is a $\bG$-martingale.

{B3.} $\cG_0$ is trivial (so $\cF_0$ is also trivial).

{B4.} The predictable projection $\ ^{p, \bF}\! \left( \frac{d\langle \wh{M}^\bG, Z\rangle^\bG_t}{d\langle Z\rangle^\bG_t}\right)$ exists for each $t\geq 0$.

\medskip

\begin{remark}\label{rem:immersion}
	Note that, under B1 and B2, the  immersion property holds between $\ff$ and $\gg$. Thus, we have $$\langle Z\rangle = \langle Z\rangle^\bG,$$ where, we recall, $\langle Z\rangle$ denotes the $\bF$--predictable quadratic variation process of $Z$.  Consequently, $[Z]-\langle Z\rangle $ is a $\bG$--martingale.
\end{remark}

In order to proceed, we denote by $\lambda_{\langle Z\rangle}$ a measure on $\cF  \otimes \cB(\bR_+)$ defined by
\[
\lambda_{\langle Z\rangle} (A) = \bE\Big( \int_{[0,\infty[}\1_{A}(\cdot , s ) d {\langle Z\rangle_s}(\cdot) \Big), \quad \text{ for } A \in \cF  \otimes \cB(\bR_+).
\]

In particular, for any $\cF  \otimes \cB(\bR_+)$-measurable function $f(\omega,s)$ we have
\[\int_{\Omega \times \bR_+} f(\omega,s)d\lambda_{\langle Z\rangle}(\omega,s)=\bE\Big( \int_{[0,\infty[}f_s d {\langle Z\rangle_s}\Big),\]
where $f_s(\omega):=f(\omega,s)$, if the expression on the right hand side is well defined. The above relation gives definition of integrability with respect to the measure $\lambda_{\langle Z\rangle} $.

The following theorem presents computation of the  $\bF$-characteristics of  $^{o,\bF}X$.

\begin{theorem}\label{not-adapted}
	Let X be a special $\bG$-semimartingale with $\bG$-characteristic triple $(B^\bG ,C^\bG,\nu^\bG)$. Assume that $\hat{\textrm{A1}}$ -- $\hat{\textrm{A3}}$ and B1 --   B4 are satisfied.
	Then the optional projection $^{o,\bF}X$ exists, it is an  $\bF$-special semimartingale and its $\bF$-characteristics are
	\begin{equation}\label{eq:BF}
	B^\bF_\cdot =  \int_0^\cdot {^{o,\bF}}\Big(b^{\bG}_s a^{\bG}_s + \int_{|x| >1} x K^\bG_s(dx)a^\bG_s \Big) ds  - \int_0^\cdot \int_{|x| >1} x  \nu^\bF(ds,dx),
	\end{equation}
	\begin{equation}\label{eq:CF}
	C^\bF_\cdot = \int_0^\cdot h^2_s d \langle Z^c\rangle _s,
	\end{equation}
	\begin{equation}\label{eq:nuF}
	\nu^\bF(A,dt) = \int_{\bR} \1_{A\setminus \set{0}}(h_t x) \nu^{Z,\bF}(dx,dt), \qquad A \in \mathcal{B} (\bR),
	\end{equation}
	where $\nu^{Z,\bF}$ is the $\ff$-compensator of the jump measure of $Z$ and
	\begin{equation}\label{eq:h}
	h_t= \ ^{p, \bF}\! \left( \frac{d\langle \wh{M}^\bG, Z\rangle^\bG_t}{d\langle Z\rangle^\bG_t}\right) = \bE\left ( \frac{d\langle \wh{M}^\bG, Z\rangle^\bG_t}{d\langle Z\rangle^\bG_t}\bigg\vert \F_{t-}  \right )=\bE\left ( \frac{d\langle \wh{M}^\bG, Z\rangle^\bG_t}{d\langle Z\rangle_t}\bigg\vert \F_{t-}  \right ) \qquad  \lambda_{\langle Z\rangle}  \-- a.e.
	\end{equation}
\end{theorem}
\begin{proof} 	
	%
	Step 1.
	First we will show that $^{o,\bF} X$ exists and is an $\bF$-special semimartingale. Towards this end it suffices to show that for every $\bF$-stopping time $\tau$ random variable $X_\tau \1_{\tau < \infty }$ is $\sigma$-integrable with respect to $\cF_\tau$. Let us take $\bF$ stopping time $\tau$  and consider increasing sequence of sets $A_n := \set{\tau \wedge n =\tau} \in \cF_{\tau \wedge n}$.
	Note that by \eqref{eq:ass-AG} we have
	\[
	\wh{B}^\bG_t =  \int_0^t \wh{b}^\bG_u a^\bG_u du.
	\]
	Using above formula, $\hat{A1}$ and $\hat{A3}$, i.e. $\bG$-martingale property of $\wh{M}^\bG$, we obtain
	\begin{align*}
	\bE( X_\tau \1_{\tau<\infty} \1_{A_n})
	&=
	\bE( X_{\tau \wedge n} \1_{A_n})
	=
	\bE\Big( (X_0 + \wh{M}^\bG_{\tau \wedge n} + \wh{B}^\bG_{\tau \wedge n})  \1_{A_n}\Big)
	\\
	&=
	\bE\Big( X_0\1_{A_n} + \bE(\wh{M}^\bG_{n} \1_{A_n}| \cG_{\tau \wedge n})  + \1_{A_n} \int_0^{\tau \wedge n} \wh{b}^\bG_u a^\bG_u du  \Big)
	\\
	&=
	\bE\Big( X_0\1_{A_n} + \wh{M}^\bG_{n} \1_{A_n}  + \1_{A_n} \int_0^{\tau \wedge n} \wh{b}^\bG_u a^\bG_u du  \Big)
	\\
	&\leq
	\bE\Big( X_0\1_{A_n} + |\wh{M}^\bG_{n}| \1_{A_n}  + \1_{A_n} \int_0^{n} |\wh{b}^\bG_u| a^\bG_u du  \Big) < \infty.
	\end{align*}
	Thus using \cite[Theorem 5.1]{HeWanYan1992} we conclude that $^{o,\bF}X$ exists. Note that as a by-product of the above estimate  we also get that $X_0$, $\wh{M}^\bG$ and $\wh{B}^\bG$ are $\sigma$-integrable with respect to  $\cF_\tau$ for every $\bF$-stopping time $\tau$ and hence using again \cite[Theorem 5.1]{HeWanYan1992} we conclude that $^{o,\bF}\wh{M}^\bG$ and $^{o,\bF}\wh{B}^\bG$
	exist. By linearity of $\bF$-optional projections and assumptions B3 and $\hat{A2}$, we may now write
	\begin{align}\nonumber
	^{o,\bF} X_t &= \ ^{o,\bF}\!X_0 + ^{o,\bF}\!\wh{M}^\bG_t+^{o,\bF}\! \wh{B}^\bG_t
	\\ \label{eq:important-decompostion}
	&=  X_0 + ^{o,\bF}\!\wh{M}^\bG_t+^{o,\bF}\! \wh{B}^\bG_t  - \int_0^t\,  {^{o,\bF}}\! (\wh{b}^\bG a^\bG)_udu +  \int_0^t\,  {^{o,\bF}} (\wh{b}^\bG a^\bG)_udu,
	\end{align}
	The process  $\wh{M}^B_t={^{o,\bF} \wh{B}^\bG_t}- \int_0^t\,  {^{o,\bF}} (\wh{b}^\bG a^\bG)_udu$  is an  $\bF$-martingale (see Lemma \ref{lem:heroic-result} and Remark \ref{lem:heroic-remark}). Invoking assumptions B1 and B2, which, in fact, imply the immersion between $\bF$ and $\bG$,  and recalling Lemma \ref{lem:heroic-result} again we see that this process is null.
	Hence we conclude that
	\begin{equation}\label{eq:Fspecdec}
	^{o,\bF} X_t =X_0+\,^{o,\bF}\!\wh{M}^\bG_t+\int_0^t \,{^{o,\bF}} (\wh{b}^\bG a^\bG)_s ds
	\end{equation}
	The process $^{o,\bF}\wh{M}^\bG$ is an $\bF$-martingale, since for an arbitrary bounded $\bF$-stopping time $\tau$ we have
	\[
	\bE( ^{o,\bF} \wh{M}^\bG_\tau ) = \bE( \bE ( \wh{M}^\bG_\tau | \cF_\tau  ))
	=
	\bE( \wh{M}^\bG_\tau )
	=
	\bE( \wh{M}^\bG_0 )
	=0.
	\]
	Moreover, the  process $\int_0^\cdot \,{^{o,\bF}} (\wh{b}^\bG a^\bG)_s ds$ is an $\bF$-predictable process with finite variation.
	From this and from \eqref{eq:Fspecdec} we deduce  that the process $^{o,\bF} X$ is an $\bF$-special semimartingale.
	Hence, using again  \eqref{eq:Fspecdec},  by uniqueness of canonical decomposition of $^{o,\bF}\!X_t = X_0+\wh{M}^\bF_t + \wh{B}^\bF_t$, we have
	\begin{equation}\label{eq:MF=FMG}
	\wh{M}^\bF_t =  {^{o,\bF}\!\wh{M}^\bG_t}, \qquad
	\wh{B}^\bF_t = \int_0^t \,{^{o,\bF}} (\wh{b}^\bG a^\bG)_s ds.
	\end{equation}
	
	Step 2. Now we will compute the $\bF$-characteristics of $^{o,\bF} X$.
	
	The formulae \eqref{eq:MF=FMG}, \eqref{eq:def:hatbG} and \eqref{eq:modBF}
	imply  that the first characteristic of $^{o,\bF} X$, that is $B^\bF $, is given by \eqref{eq:BF}.
	
	Now, since $\wh{M}^\bG$ is square integrable then, invoking the Jensen inequality, we conclude that $\wh{M}^\bF =  {^{o,\bF}\!\wh{M}^\bG}$ is square integrable. Next, invoking the predictable representation property we see that there exists an $\bF$-predictable process $h$ such that $\bE \int_0^t h^2_s d [ Z ]_s < \infty $ and  \begin{equation}\label{eq:MFrep}
	\wh{M}^\bF_t=\int    _0^ t h_sdZ_s .
	\end{equation}
	Now we will compute remaining characteristics of $^{o,\bF}X$ in terms of the process $h$ and in the following step we will compute $h$.
	The continuous martingale part of $^{o,\bF}X$ is thus given as $\int_0^ t h_sdZ^c_s$, where $Z^c$ is the continuous part of the $\bF$-martingale $Z$, so that
	\[
	C^\bF_t= \int_0^t h^2_s d \langle Z^c \rangle_s.
	\]
	Here, $Z^c$ being continuous, $\langle Z^c \rangle=\langle Z^c \rangle^ \bG$.
	To complete this step of  the proof we need to justify \eqref{eq:nuF}. This formula is a consequence of  the fact that $\Delta ^{o,\bF} X_t = h_t \Delta Z_t$,  which entails  that
	the jump measure of $^{o,\bF} X$ is the image of the jump measure of $Z$  under the  mapping $ (t,x) \rightarrow (t, x h_t  \1_{\set{h_t \neq 0}})$, and thus the $\bF$-compensator of $^{o,\bF} X$ is the image of the $\bF$-compensator of $Z$.
	
	Step 3.
	We will now compute  $h$. Towards this end, we fix $t\geq 0$ and we observe using \eqref{eq:MF=FMG} that  for any bounded  $\F_t$-measurable  random variable $\gamma$  we have
	\be \label{eq:carach}
	\bE(\gamma \wh{M}^ \bG_t)=\bE(\gamma \wh{M}^\bF_t).
	\ee
	By using  integration by parts formula we may write the left-hand side of \eqref{eq:carach} as
	\begin{equation}\label{eq:LHS}
	\bE(\gamma \wh{M}^ \bG_t )=\bE \Big( \int_0^t \widehat \gamma_{s-} d \wh{M}^\bG_s  + \int_0^t  \wh{M}^\bG_{s-} k_s d Z_s  + [ \widehat \gamma, \wh{M}^\bG ]_t \Big),
	\end{equation}
	where $(\widehat \gamma_s)_{ s\in[0,t]}$  is the  bounded martingale defined by $ \widehat \gamma_s := \bE (  \gamma | \cF_s)$ which admits the representation
	\begin{equation}\label{gamma}
	\widehat \gamma_s   =  \bE ( \gamma) + \int_0^s k_u d Z_u,\quad s\in[0,t],
	\end{equation} from which $k$ is obtained.
	In view of assumption $\hat{A3}$, we know that the process  $\wh{M}^ \bG$ admits a Kunita-Watanabe decomposition of the form
	\begin{equation}\label{emik}
	\wh{M}^ \bG_t= \wh{M}^ \bG_0 + \int_0^ t H_sdZ_s+\wh{O}^\bot_t,
	\end{equation}
	where $\wh{O}^\bot$ is a square integrable $\bG$-martingale orthogonal to $Z$ satisfying $\wh{O}^\bot_0=0$, and $H$ is a $\bG$-predictable process such that $\int_0^ \cdot H_sdZ_s$ is a square integrable $\bG$-martingale (see e.g. \cite{Sch2001}).
	Hence, since $Z$ is assumed to be square integrable,  we have
	\begin{equation}\label{H} H_t=\frac{d\langle \wh{M}^\bG, Z\rangle^\bG_t}{d\langle Z\rangle^\bG_t}.
	\end{equation}
	Now, let us note that from the representation \eqref{gamma} of $\widehat \gamma$ as stochastic integral with respect to $Z$ and from  \eqref{emik}  we may write
	\[ [ \widehat \gamma, \wh{M}^\bG ]_t = \int_0^t k_s d [Z, \wh{M}^\bG]_s = \int_0^t k_s \Big(H_s d [Z]_s + d [Z, \wh{O}^\bot]_s \Big).
	\]
	Using this we obtain from \eqref{eq:LHS}
	\begin{align}\label{eq:mtgs}
	\bE( \gamma \wh{M}^\bG_t )&= \bE \Big( \int_0^t \widehat \gamma_{s-} d \wh{M}^\bG_s   + \int_0^t  \wh{M}^\bG_{s-} k_s d Z_s + \int_0^t k_s H_s d [Z]_s + \int_0^t
	k_s  d [Z, \wh{O}^\bot]_s
	\Big).
	\end{align}
	Now we prove that the stochastic integrals $ \int_0^\cdot \widehat \gamma_{s-} d \wh{M}^\bG_s$, $\int_0^\cdot  \wh{M}^\bG_{s-} k_s d Z_s$  and $\int_0^\cdot
	k_s  d [Z, \wh{O}^\bot]_s$ in \eqref{eq:mtgs} are $\bG$-martingales on $[0,t]$. The first stochastic integral, i.e., $ \int_0^\cdot \widehat \gamma_{s-} d \wh{M}^\bG_s$, is a $\bG$-martingale since $(\widehat \gamma_{s-})_{ s\in[0,t]}$ is a bounded predictable process. Next, we prove that the local martingale
	\begin{equation}\label{eq:XkdZ}
	\int_0^u  \wh{M}^\bG_{s-} k_s d Z_s, \qquad u \in [0,t],
	\end{equation}
	is a $\bG$-martingale. Using \cite[Lemma 16.2.5]{CohEll2015}, \cite[Theorem 16.2.6]{CohEll2015} and the \`{E}mery inequality for BMO (\cite[Theorem A.8.15.]{CohEll2015}) applied to local martingales (with $p=1$) we obtain
	\[
	\bE (\wh{M}^\bG_{-}  \!\cdot \! ( k  \!\cdot \! Z))_t ^* \leq C \bE (\wh{M}^\bG)^*_t \ \norm{ ( k  \!\cdot \! Z) }_{BMO} .
	\]
	From the Doob maximal inequality, we obtain that $\bE (\wh{M}^\bG)^*_t < \infty$. Next, since $\widehat \gamma$ is bounded, using \cite[Remark A.8.3.]{CohEll2015} we see that $\norm{ ( k  \!\cdot \! Z) }_{BMO} < \infty$. Therefore the local martingale given by \eqref{eq:XkdZ} is a martingale.
	Finally, we consider
	\begin{equation}\label{eq:kdZM}
	\int_0^u
	k_s  d [Z, \wh{O}^\bot]_s , \qquad u \in [0,t].
	\end{equation}
	Using Kunita-Watanabe's inequality and   Cauchy-Schwartz's inequality we obtain that
	\[
	\bE (k  \!\cdot \!  [Z, \wh{O}^\bot])^*_t \leq
	\bE \int_0^t |k_s|  |d[Z, \wh{O}^\bot]_s | \leq 	\left( \bE \int_0^t k^2_s  d [Z]_s \right)^{1/2}\left(\bE  [\wh{O}^\bot]_t\right)^{1/2} < +\infty.
	\]	
	So   the process given by \eqref{eq:kdZM}  is a martingale.
	Consequently, since the processes $ \int_0^\cdot \widehat \gamma_{s-} d \wh{M}^\bG_s$, $\int_0^\cdot  \wh{M}^\bG_{s-} k_s d Z_s$  and $\int_0^\cdot
	k_s  d [Z, \wh{O}^\bot]_s$ in \eqref{eq:mtgs}
	are $\bG$-martingales on $[0,t]$, the left hand side of  \eqref{eq:carach} takes the form
	\begin{equation}\label{eq:LHS-j}
	\bE( \gamma \wh{M}^ \bG_t)
	= \bE \Big( \int_0^t k_s H_s d [Z]_s \Big).
	\end{equation}
	Now we deal with the right-hand side of \eqref{eq:carach}. Invoking \eqref{eq:MF=FMG} and  \eqref{gamma}, and using  integration by parts formula,  we may write the right-hand side of \eqref{eq:carach} as
	\begin{equation}\label{eq:RHS}
	\bE( \gamma \wh{M}^\bF_t )=\bE( \gamma\, ^{o,\bF}\!\wh{M}^\bG_t )= \bE \Big( \int_0^t \widehat \gamma_{s-} d ^{o,\bF}\!\wh{M}^\bG_s  + \int_0^t  {^{o,\bF}\!}\wh{M}^\bG_{s-} k_s d Z_s  + [ \widehat \gamma, ^{o,\bF}\!\wh{M}^\bG ]_t \Big).
	\end{equation}
	Next, let us note that from \eqref{eq:MF=FMG}, \eqref{gamma} and \eqref{eq:MFrep} we get
	\[ [ \widehat \gamma, {^{o,\bF}\!\wh{M}^\bG }]_t = \int_0^t k_s d [Z, \wh{M}^\bF]_s = \int_0^t k_s h_s d [Z]_s.
	\]
	Using this and \eqref{eq:RHS} we obtain
	\begin{align}\label{eq:mtgs-j}
	\bE( \gamma \wh{M}^\bF_t )
	= \bE \Big( \int_0^t \widehat \gamma_{s-} d \wh{M}^\bF_s  + \int_0^t   \wh{M}^\bF_{s-} k_s d Z_s + \int_0^t k_s h_s d [Z]_s \Big).
	\end{align}
	Applying reasoning analogous to the one that led to \eqref{eq:LHS-j}, and invoking \eqref{eq:mtgs-j}
	we conclude that
	\begin{equation}\label{eq:RHS1}
	\bE( \gamma \wh{M}^ \bF_t )=\bE \Big(  \int_0^t k_s h_s d [Z]_s \Big).
	\end{equation}
	Putting together \eqref{eq:carach}, \eqref{eq:LHS-j} and \eqref{eq:RHS1}, we see that (\ref{eq:carach}) is equivalent to
	\begin{equation}\label{k}
	\bE\left ( \int _0^t H_sk_s d[Z]_s\right ) =\bE\left ( \int _0^t h_sk_s d[Z]_s\right )
	\end{equation}
	for any $k$ which is $\bF$-predictable and such that $\int_0^t k_sdZ_s$ is bounded.
	
	We will now show that \eqref{k} extends to any $\bF$-predictable and bounded $k$, a result that we will need in what follows.
	Towards this end let us take an arbitrary predictable and bounded $k$ and define a  square integrable random variable $\psi $ by
	\[
	\psi := \int_0^t k_s d Z_s .
	\]
	The random variable	$\psi $ is a (point-wise) limit of the sequence $\psi_n := \psi \wedge n$ of bounded random variables and hence $\bE \psi_n \rightarrow \bE \psi = 0$.
	Moreover, for each $n$ we have the predictable representation $\psi_n = \bE (\psi_n)  + \int_0^t k^n_s d Z_s$, and thus
	\begin{align*}
	\bE \Big( \int_0^t (k^n_s - k_s)^2  d[Z]_s \Big) &= \bE \Big( \int_0^t k^n_s d Z_s - \int_0^t k_s d Z_s\Big)^2 = \bE(( \psi_n - \psi - \bE (\psi_n)) ^2) \\
	& \leq 2 \bE (( \psi_n - \psi )^2) + 2 \left(\bE (\psi_n)\right )^2  \mathop{\longrightarrow}_{ n \rightarrow \infty }0 .
	\end{align*}
	Using this and the Kunita-Watanabe inequality we obtain
	\[
	\bE\left (  \int _0^t |H_s(k^n_s - k_s )| d[Z]_s \right ) \leq  \left( \bE\left(  \int _0^t |H_s|^2 d[Z]_s  \right ) \right)^{\frac{1}{2}}  \bigg( \bE\left (  \int _0^t |k^n_s - k_s |^2 d[Z]_s \right ) \bigg)^{\frac{1}{2}}		 \mathop{\longrightarrow}_{ n \rightarrow \infty }0
	\]
	and
	\[
	\bE\left (  \int _0^t |h_s(k^n_s - k_s )| d[Z]_s \right ) \leq  \bigg( \bE\left (  \int _0^t |h_s|^2 d[Z]_s \right ) \bigg)^{\frac{1}{2}}  \bigg( \bE\left (  \int _0^t |k^n_s - k_s |^2 d[Z]_s \right ) \bigg)^{\frac{1}{2}}		 \mathop{\longrightarrow}_{ n \rightarrow \infty }0.	 \]
	Using  these two facts and \eqref{k} for $k^n$, we can pass to the limit in  \eqref{k} and obtain that \eqref{k}  holds for any bounded $k$.
	
	Recall that $[Z]-\langle Z\rangle $ is a $\bG$--martingale (cf. Remark \ref{rem:immersion}). Thus, using the Kunita-Watanabe inequality  we obtain that $(\int_0^\cdot H_sk_s d  ([Z]_s - \left< Z \right>_s))$ is a $\bG$--martingale and hence
	\[
	\bE\left ( \int _0^t H_sk_s d  [Z]_s \right) = \bE\left ( \int _0^t H_sk_s
	d\left< Z \right>_s\right ) .
	\]
	Similarly
	\[
	\bE\left ( \int _0^t h_sk_s d  [Z]_s \right) = \bE\left ( \int _0^t h_sk_s
	d \left< Z \right>_s\right ) .
	\]
	Hence and  from \eqref{k}, we have
	\begin{equation*}
	\bE\left ( \int _0^t (H_s - h_s)k_s d \left< Z \right>_s\right ) = 0.
	\end{equation*}
	This implies, by using assumption B.4, \cite[Theorem 5.16]{HeWanYan1992} and the remark right below it, that
	\begin{equation*}
	\bE\left ( \int _0^t ( ^{p,\bF} H_s - h_s)k_s d \left< Z \right>_s\right ) = 0
	\end{equation*}
	for any bounded $\bF$--predictable $k$ such that $(H_t - h_t)k_t\geq 0$ for $t\geq 0$. Hence, we have \[\int_{\Omega \times \bR_+} \1_{(0,t]}(s)( ^{p,\bF} H_s(\omega) - h_s(\omega))k_s(\omega)d\lambda_{\langle Z\rangle}(\omega,s)
	=0\] for any bounded $\bF$--predictable $k$ such that $(H_t - h_t)k_t\geq 0$ for $t\geq 0$. Since, by convention, semi-martingales are right continuous, then the measure  $\lambda_{\langle Z\rangle}$ does not charge any set of the form $B\times \{0\}$. Consequently, we conclude that \[\int_{\Omega \times \bR_+} \1_{[0,t]}(s)( ^{p,\bF} H_s(\omega) - h_s(\omega))k_s(\omega)d\lambda_{\langle Z\rangle}(\omega,s) = 0\] for any bounded $\bF$--predictable $k$ such that $(H_t - h_t)k_t\geq 0$ for $t\geq 0$.
	
	Now, note that $H_s-h_s >0$ if and only if $^{p,\bF}H_s-h_s>0$.
	Thus, first taking
	\[
	k_s = \1_{\set{ ( H_s-h_s)  > 0}} =  \1_{\set{ ({^{p,\bF}H_s} -h_s)  > 0}},
	\]
	and then taking
	\[
	k_s = -\1_{\set{ (H_s-h_s)  < 0}} = -\1_{\set{ ({^{p,\bF}H_s} -h_s)  < 0}}
	\]
	in the above, we obtain
	\begin{equation}\label{eq:H-to-h}
	h_s=\ ^{p,\bF}\!H_s \qquad  \lambda_{\left< Z \right>}- \text{  a.e.   on } \Omega \times  [0,t].
	\end{equation}
	This, together with formula \eqref{H} gives \eqref{eq:h}.
	The proof is complete.
\end{proof}

\section{Examples}   \label{sec:examples}

\def \ind{1\!\!1}
\def \E{\mathbb E}

Examples \ref{vopp}--\ref{ex:BivDiff} below illustrate the results in the case when $X$ is $\bF$-adapted. In what follows the natural filtration of any  process $A$ is denoted by $\bF^A$.

\begin{example} \label{vopp}

	Consider two one-point c\`{a}dl\`{a}g  processes $Y^1$ and $Y^2$ on a probability space $(\Omega,\F,\bP)$, and let  $Y=(Y^1,Y^2)$.  That is,  $Y^i$ $(i=1,2)$  starts from $0$ at time $t=0$ and jumps to $1$ at some random time.  Thus, $Y$ can be identified with a pair of positive random variables $T_1$ and
	$T_2$ given by $T_i := \inf\{t>0:\ Y^i_t =1\}$, $i=1,2$. In other words, $Y^i_t=\1_{\{T_i \leq t\}}$, $i=1,2$. We assume that, under $\bP$, the probability distribution of $(T_1,T_2)$ admits a density function $f(u,v)$ which is continuous in both variables.
	
	Now, let $X=Y^1$, $\bF=\bF^{X}$  and $\bG=\bF^{Y}$. Clearly, $X$ is a special $\bG$-semimartingale and a special $\bF$-semimartingale on $(\Omega,\F,\bP)$.
	
	The $\bG$-characteristics of $X$ are $(B^\bG,0,\nu^\bG)$, where
	\beq
	B^\bG_t=\int_0^t\, \kappa_sds,\quad \nu^\bG(ds,dx)=\delta_{1}(dx) \kappa_s \,ds,\eeq
	$\delta_1$ is the Dirac measure at $1$, and $\kappa$ is given by (this result follows, for example, by application of \cite[Theorem 4.1.11]{LasBra1995})
	\[
	\kappa_s = \frac{\int_s^{\infty}f(s,v)\,dv}{\int_s^\infty
		\int_s^\infty f(u,v)\,du\,dv}\1_{\{s \leq T_1 \wedge T_2\}} +
	\frac{f(s,T_2)}{\int_s^{\infty} f(u,T_2)\,du}\1_{\{T_2<s\leq
		T_1\}}, \quad s\geq 0.
	\]

	Thus, according to Theorem \ref{adapted}, the $\bF$-characteristics of $X$ are $(B^\bF,0,\nu^\bF)$, where
	\[B^\bF_t=\int_0^t{^{o,\bF}}(\kappa)_sds,\quad \, \nu^{\bF}=(\nu^\bG)^{p,\bF}.\]
	Now, we will provide  explicit formulae for $B^\bF$ and $\nu^{\bF}$; for the latter, we only need to compute $\nu^{\bF}(dt,\{1\})$.
	It can be easily shown that these computations  boil down to computing the $\bF$-optional projection of the process $\kappa$. Indeed, for an arbitrary $\bF$-predictable, bounded function $W$ on $\Omega \times \bR$ we have
	\begin{align*}
	&\bE\Big( \int_{\bR_+ \times \bR } W(s,x) \nu^{\bG} (ds, dx)\Big)
	=
	\bE\Big(\int_{\bR_+ } W(s,1) \kappa_s ds\Big)
	=
	\bE\Big(\int_{\bR_+ } {^{p,\bF}}(W(\cdot , 1) \kappa_\cdot)_s ds\Big) \\
	&=
	\bE\Big(\int_{\bR_+ }  {^{p,\bF}}(\kappa)_s W(s , 1) ds\Big) =
	\bE\Big(\int_{\bR_+ \times \bR } {^{p,\bF}}(\kappa)_s W(s , x) \delta_1(dx) ds\Big),
	\end{align*}
	where ${^{p,\bF}}(\kappa)$  denotes the  $\bF$-predictable projection of $\kappa$. Next, we note that the measure $\rho$ defined as
	\[
	\rho (dt, dx) := {^{p,\bF}}(\kappa)_t\delta_1(dx)  dt
	\]
	is $\bF$-predictable, and thus, due to uniqueness of the dual predictable projections, we have  $\rho = (\nu^\bG)^{p, \bF}$, and so $\nu^\bF = \delta_1(dx) \, {^{p,\bF}}(\kappa)$.
	Finally, we note that, in view of the continuity assumptions on $f$ and that fact that   $\kappa$ admits two jumps only, we have
	\[
	\bE\Big(\int_{\bR_+ \times \bR } {^{p,\bF}}(\kappa)_s W(s , x) \delta_1(dx) ds \Big)=\bE\Big(\int_{\bR_+ \times \bR } {^{o,\bF}}(\kappa)_s W(s , x) \delta_1(dx)  ds\Big),
	\]
	where ${^{o,\bF}}(\kappa)$  denotes the  $\bF$-optional projection of $\kappa$.
	Using  the key lemma (see e.g. \cite[Lemma 2.9]{AksJea2017}) we obtain
	\beq
	{^{o,\bF}}(\kappa)_s
	&=&\E\left(\frac{\int_s^{\infty}f(s,v)\,dv}{\int_s^\infty
		\int_s^\infty f(u,v)\,du\,dv}\1_{\{s \leq T_1 \wedge
		T_2\}}+
	\frac{f(s,T_2)}{\int_s^{\infty}
		f(u,T_2)\,du}\1_{\{T_2<s\leq
		T_1\}}\bigg|\F_{s}\right)
	\\
	& =&\frac{\int_0^{\infty}f(s,v)\,dv}{\int_s^\infty
		\int_0^\infty f(u,v)\,du\,dv} \1_{\set{ T_1 > s }} .
	\eeq
	Consequently,
	\[B^\bF_t=\int_0^t  \frac{\int_0^{\infty}f(s,v)\,dv}{\int_s^\infty
		\int_0^\infty f(u,v)\,du\,dv} \1_{\set{ T_1 > s }} ds \]
	and  ${\nu}^\bF((0,t],\{1\})$ is given as
	\beq
	&&{\nu}^\bF((0,t],\{1\})=
	\int_0^t \frac{\int_0^{\infty}f(s,v)\,dv}{\int_s^\infty
		\int_0^\infty f(u,v)\,du\,dv} \1_{\set{ T_1 > s }} ds.
	\eeq
	We note that the last result agrees with the classical computation of intensity of $T_1$ in its own filtration, which is given as $\lambda^1 _s= \frac{f^1(s)}{1-F^1(s)}$ with $F^1(s)=\bP(T_1\leq s)$ and $f^1(s)=\frac{\partial F^1(s)}{\partial s}$.
	\qed
\end{example}

\begin{example}\label{diff}
	Let $X$ be a real-valued process on $(\Omega,\cF,\bP)$ satisfying
	\[dX_t=m_tdt +\sum_{j=1}^2 \sigma^{j}_tdW^j_t +  dM_t,\quad t\geq 0,\]
	where $W^j$s are independent standard Brownian motions (SBMs), and $M_t=\int_0^t \int_\bR x(\mu(ds,dx)-\nu(ds,dx))$ is a pure jump martingale, with absolutely continuous compensating part, say $\nu(dx,dt)=\eta(t,dx)dt$. We assume that $M$  is independent of $W^j$s. The coefficients $m$ and $\sigma^j>0,\ j=1,2$ are adapted to $\bG:=\bF^{W^1,W^2,M }$ and bounded.
	
	\noindent
	Let    $\bF=\bF^{X}$.  Since $M$ and  $\sigma ^1\cdot W^1+\sigma^2\cdot W^2$ are true $\bG$-martingales,  then $X$ is a special semimartingale in $\bG$ and thus in $\bF$.
	
	The $\bG$-characteristics of $X$ are
	\[
	B^\bG_t= \int_0^t  m_sds, \quad C^\bG_t=\int_0^t((\sigma^1_s)^2+(\sigma^2_s)^2)ds,\quad \nu^\bG(dx,dt)=  \eta(t,dx) dt.
	\]
	Now, in view of Theorem \ref{adapted},  we conclude that the $\bF$-characteristics of $X$ are
	\[
	B^\bF_t=\int_0^t {^{o,\bF}}(m)_sds  , \quad C^\bF_t=\int_0^t((\sigma^1_s)^2+(\sigma^2_s)^2)ds,\quad \nu^\bF(dx,dt)= (\eta(t,dx) dt)^{p,\bF}. \qquad \qed
	\]
\end{example}

\begin{example}\label{ex:PP}
	In this example we consider  time homogeneous Poisson process with values in $\r^2$.
	There is a one-to-one correspondence between any time homogeneous Poisson process with values in $\r^2$, say  $N=(N^1,N^2)$, and a homogeneous Poisson measure, say $\mu$, on $E:=\{0, 1\}^2 \setminus \{(0,0)\}$.\footnote{We refer to \cite{JacShi2002} for the definition of the Poisson measure.} See for instance discussion in \cite{BieJakVidVid2008}.
	
	Let $\bG=\bF^N$, and let $\nu$ denote the
	$\bG$-dual predictable projection of  $\mu$.  The measure $\nu$ is a
	measure on a finite set, so it is uniquely determined by its values
	on the atoms in $E$. Therefore the Poisson process $N=(N^1,N^2)$ is
	uniquely determined by
	\be\lab{pmeas}
	\nu(dt,\{1,0\})=\lambda_{10}
	dt,\quad
	\nu(dt,\{0,1\})=\lambda_{01}
	dt,\quad
	\nu(dt,\{1,1\})=\lambda_{11}
	dt\ee
	for
	some positive constants $\lambda_{10}$, $\lambda_{01}$ and
	$\lambda_{11}$. Clearly, the Poisson process $N=(N^1,N^2)$  is a $\bG$-special semimartingale, and the $\bG$-characteristic triple of $N$  is $(B,0,\nu)$, where
	
	\[B_t={(\lambda_{10}+\lambda_{00})t \atopwithdelims []
		(\lambda_{01}+\lambda_{00})t}.\]
	
	Let $X=N^1$. Then, $X$ is a $\bG$-special semimartingale, and the $\bG$-characteristic triple of $X$  is $(B^\bG,0,\nu^\bG)$, where
	\beq \nu^\bG(dt,\set{1})&=&\nu(dt,\{(1,0)\}) +
	\nu(dt,\{(1,1)\}) =\lambda_{10}dt + \lambda_{11}dt, \qquad \nu^\bG(dt,\set{0}) = 0,
	\eeq  and  $B^\bG_t=(\lambda_{10}+\lambda_{00})t$.
	
	Now, let us set $\bF=\bF^{X}$. In view of Proposition \ref{prop:det-char} we have
	\[(B^\bF,0,\nu^\bF)=(B^\bG,0,\nu^\bG).\]
	\qed
\end{example}

\begin{example}\label{ex:BivDiff}
	
	Let $Y=(Y^1,Y^2)^\top$ be
	given as the strong solution of the SDE
	\be\lab{joint}
	dY_t=m(Y_t)dt +\Sigma (Y_t)dW_t,\quad Y(0)=(1,1)^\top,\,
	\ee
	\noindent where $W=(W_1,W_2)^\top$ is a two dimensional SBM process on  $(\Omega,\F,\bP)$, and where $$m(y^1,y^2)=(m_1(y^1,y^2),m_2(y^1,y^2))^\top, \quad  \Sigma(y^1,y^2)=\left(
	\begin{array}{cc}
	\sigma_{11}(y^1,y^2) & \sigma_{12}(y^1,y^2) \\
	\sigma_{21}(y^1,y^2) & \sigma_{22}(y^1,y^2) \\
	\end{array}
	\right)$$
	are bounded.
	Next, let us set {$\bG=\bF^Y$},  $X=Y^1$ and  $\bF=\bF^X$. Hence
	
	\begin{equation}\label{eq:Xdyn}
	dX_t=m_1(X_t, Y^2_t)dt +\sigma_{1,1} (X_t, Y^2_t)dW^1_t + \sigma_{1,2} (X_t, Y^2_t)dW^2_t
	\end{equation}
	
	\noindent
	Suppose that the function $\Sigma $ satisfies the following  condition
	\[
	\sigma ^2_{11}(y^1,y^2)+\sigma ^2_{12}(y^1,y^2)
	=\sigma^2_1(y^1),\quad (y^1, y^2) \in \bR^2,
	\]
	for some function $\sigma_1 >0$, and
	suppose that the function $m_1$
	satisfies
	\[
	m_1(y^1,y^2)=\mu_1(y^1) \quad (y^1, y^2) \in \bR^2.
	\]
	Then \eqref{eq:Xdyn}  takes form
	\begin{equation*}
	dX_t=\mu_1(X_t)dt +\sigma_{1} (X_t)dZ_t, \qquad X(0) =1,
	\end{equation*}
	where
	\[
	Z_t = \int_0^t \frac{\sigma_{1,1} (X_s, Y^2_s) }{\sigma_1(X_s)} d W^1_s + \int_0^t  \frac{\sigma_{1,2} (X_s, Y^2_s) }{\sigma_1(X_s)} d W^2_s
	\]
	is a $\bG$-adapted process, which is a continuous $\bG$-local martingale. Since $(Z^2(t) - t)_{t \geq 0}$  is a local martingale we obtain by L\'evy's characterization theorem that $Z$ is a standard Brownian motion in the filtration $\bG$. Thus using continuity of paths of $X$ we conclude that $X$
	has the  $\bG$-characteristic triple given as $(B^\bG,C^\bG,0)$, where
	\[
	{B^\bG_t}=\int_0^t \mu_1(X_u) du,\quad  {C^\bG_t}=\int_0^t \sigma^2_1(X_u) du, \quad t\geq 0.
	\]
	We will now apply Theorem \ref{adapted} so to compute the $\bF$-characteristics of $X$. Since $X$ is $\bF$-adapted the $\bF$-characteristics of $X$ are
	\[
	{B^\bF_t}=\int_0^t
	{^{o,\bF}}(\mu_1(X_u) ) du = \int_0^t \mu_1(X_u)  du
	,
	\quad
	{C^\bF_t} = {C^\bG_t}=\int_0^t \sigma^2_1(X_u) du
	\]
	Finally by continuity of paths
	\[
	\nu^{\bF}(dt,dx) = (\nu^{\bG}(dt,dx))^{p,\bF} = (0)^{p,\bF}  	= 0 .
	\]
	So we conclude that  $( B^\bG, C^\bG,0)=(B^\bF,C^\bF,0)$.
	\qed
\end{example}

The remaining examples refer to the case when $X$ is not $\bF$-adapted. We begin with providing, in Remark \ref{rem:4.1},  a sufficient condition for the $\bF$-optional and the $\bF$-predictable projections of a $\bG$-adapted process $Y$ to exist. We will only use this condition to deduce existence of $\bF$-predictable projections though.
\begin{remark}\label{rem:4.1}
	The $\bF$-optional and $\bF$-predictable projections of a $\bG$-adapted process $Y$ exist  under the condition that $\bE (Y^*_t) < \infty$ for all $t>0$, where $Y^*_t= \sup _{s\leq t} |Y_s |$.
	Indeed, taking an  arbitrary $\bF$ stopping time and $A_n= \set{\tau \leq n},\ n=1,2,\ldots  $ we have $ A_n  \in \cF_{\tau-}\subseteq \cF_{\tau}$
	and
	\[
	\bE(Y_\tau\1_{\tau < \infty}\1_{A_n})
	\leq  \bE(Y^*_n 1_{\tau < \infty}\1_{A_n})  \leq  \bE(Y^*_n ) <\infty,
	\]	
	for each $n=1,2,\ldots  $.  Hence by \cite[Theorem 5.1]{HeWanYan1992} the $\bF$-optional projection of $Y$ exists.  Taking $\tau$ to be an $\bF$-predictable stopping time we have by \cite[Theorem 5.2]{HeWanYan1992} that the $\bF$-predictable projection of $Y$ exists.
\end{remark}

\begin{example}\label{ex:Privault}
	Let  $(\Omega,\F,\bP)$ be the underlying probability space supporting a Brownian motion $W$ and an independent time inhomogeneous Poisson process $N$. Let $\gg$ be the filtration generated by $W$ and $N$.
	Suppose  that $N$ has  deterministic compensator
	$\nu(t) = \int_0^t \lambda (s) ds$, $t\in\r_+$, so that $(\nu(t))_{t\in\r_+}$ is the unique continuous deterministic
	function such that
	\[
	M_t:=N_{t}-\nu(t), \quad t \geq 0,
	\] is an $(\bF^N,\bP)$-martingale.
	Additionally, suppose that $\lambda$ is such that  $\lim_{t\rightarrow \infty}
	\nu(t) = \infty$ and $\nu(t) <\infty$, $t\geq 0$.
	Finally, note that by independence of  $W$ and $N$ under $\bP$, $M$ is also a $(\gg,\bP)$-martingale.
	
	Let now  $X$ be a process with  the following integral  representation
	\begin{equation}\label{eq:exampledefX}
	X_t =X_0+\int_0^t \beta_s ds +  \int_0^t \gamma_s dW_s + \int_0^t \kappa_s \1_{|\kappa_s| \leq 1} d M_s + \int_0^t \kappa_s \1_{|\kappa_s| > 1} d N_s,
	\end{equation}
	for some $\bG$-predictable processes $\gamma $ and $\kappa$
	such that  for all $t\geq 0$
	\begin{equation}\label{eq:intc1}
	\bE \Big( \sup_{s \leq t}\Big(|\beta_s|
	+|\kappa_s \lambda(s)|\1_{|\kappa_s| > 1} \Big) \Big) < \infty,
	\end{equation}
	\begin{equation}\label{eq:intc3}
	\bE \Big( \int_0^t \big( \gamma_s^2 + \kappa_s^2 \lambda(s) \big) ds \Big) < \infty.
	\end{equation}
	Note that $L_t=\int_0^t \gamma_s dW_s + \int_0^t \kappa_s \1_{|\kappa_s| \leq 1} d M_s$ and $A_t=\int_0^t \beta_s ds+\int_0^t \kappa_s \1_{|\kappa_s| > 1} d N_s,\ t\geq 0$ are a $\gg$-adapted local-martingale and $\gg$-adapted process with locally integrable variation, respectively. Thus for $A^\bG_t = t$ the process $X$ is a special $\gg$-semimartingale  with $\bG$-characteristics $(B^\bG,C^\bG,\nu^\bG)$, where
	\[
	B^{\bG}_\cdot =\int_0^\cdot  \beta_t  dt,
	\]
	\[C^\bG_\cdot=\int_0^\cdot \gamma^2_tdt,\]
	and
	\[
	\nu^\bG(A,dt) = \Big(\int_\bR \1_{A \setminus \set{0}}(x) \delta_{\kappa_t}(dx) \Big)\lambda(t)dt, \qquad A \in \mathcal{B} (\bR).
	\]
	In particular, we have (cf. \eqref{data})
	\[
	a^\bG_t = 1, \quad b^\bG_t = \beta_t, \quad c^\bG_t = \gamma_t^2, \quad K^\bG_t(dx) =  \delta_{\kappa_t}(dx)  \1_{\kappa_t \neq 0 } \lambda(t).
	\]
	Moreover, $X$ as a special $\bG$-semimartingale has the unique canonical decomposition
	\[
	X  = X_0 + \wh{B}^\bG + \wh{M}^\bG,
	\]
	where
	\begin{equation}\label{eq:BGMG}
	\wh{B}^\bG_t = \int_0^t (\beta_s + \1_{|\kappa_s| > 1}\kappa_s \lambda(s) )ds \quad \text{ and } \quad \wh{M}^\bG_t = \int_0^t \gamma_s dW_s + \int_0^t \kappa_s d M_s.
	\end{equation}
	Now, we will verify that for arbitrary $\bF \subseteq \bG$ assumptions
	$\hat{\textrm{A1}}$ -- $\hat{\textrm{A3}}$ are satisfied.
	Since $a^\bG=1$ and
	\[
	\wh{b}^\bG = \beta + \1_{|\kappa| > 1}\kappa \lambda,
	\]
	from \eqref{eq:intc1} it immediately follows that $\hat{\textrm{A1}}$ holds. Note that
	\eqref{eq:intc1}  also implies $\sigma$-integrablilty of $\wh{b}^\bG a^\bG$ with respect to $\cF_\tau$ for every bounded $\bF$-stopping time $\tau$, so $\hat{\textrm{A2}}$ is satisfied.
	Using \eqref{eq:intc3} we see that
	$\wh{M}^\bG$ is a square integrable $\bG$-martingale and hence $\hat{\textrm{A3}}$ follows.
	
	\medskip
	In what follows we will
	give the form of characteristics of $^{o,\bF}\!X$ for different specifications of $\bF$ and $Z$.
	
	\medskip
	
	\noindent a) Let $\phi : \r_+\longrightarrow \r$ and
	$\alpha : \r_+\longrightarrow (0,\infty)$ be two  deterministic functions, with $\int_0^t \alpha^2 (s) ds < \infty, \ t \geq 0$.
	Let $i (t) = \1_{\{\phi(t) = 0\}}$,  and set
	$$\lambda(t)
	=
	\left\{ \begin{array}{ll}
	\alpha^2(t)/\phi^2(t)& \mbox{if} \ \phi(t)\not=0, \\
	0 & \mbox{if} \ \phi(t)=0, \ \ \
	\end{array} \right.
	t \geq 0.
	$$
	By Proposition 4 in \cite{EMM1989} the process $V$ given by
	\begin{equation}\label{eq:SE-sol}
	d V_t = i(t) d W_t + \frac{\phi(t)}{\alpha(t)}(d N_t - \lambda(t) dt ),\quad t\geq 0, \quad V_0 = 0,
	\end{equation}
	is the unique strong solution of the following structure equation 
	\begin{equation}\label{eq:SE}
	d [V]_t = dt + \frac{\phi(t)}{\alpha(t)} d V_t,\quad t\geq 0, \quad V_0 = 0.
	\end{equation}
	By Proposition 3 ii) in  \cite{EMM1989}  the process $V$ has the predictable representation property in $\bF^V$.
	
	We now take $\bF = \bF^V\subseteq \bG$ and $Z$ given as \[ Z_t =\int_0^t\alpha(s) d V_s,\ t\geq 0. \]
	Thus, following \cite{jp:cm}, we see that $Z$ satisfies
	\begin{equation}
	\label{e5}
	dZ_{t} = i(t) \alpha(t) dW_{t}
	+ \phi(t) \left( dN_{t} - \lambda(t)dt\right),
	\ \ \ t\geq 0, \ Z_0=0.
	\end{equation}
	Thus, using the fact that
	\[
	\phi(t)dZ_t= \phi^2(t)(dN_t-\lambda(t) dt) \quad  \text{and} \quad  i^2(t)\alpha^2(t)+ \phi^2(t)\lambda(t)= \alpha^2(t),
	\] we conclude that
	\begin{equation}\label{eq:d[Z]}
	d[Z]_t = i^2(t) \alpha^2(t) d{t} + \phi^2(t)dN_t = \alpha^2(t) dt +  \phi(t) d Z_t,\quad t\geq 0, \ Z_0=0.
	\end{equation}
	The process $Z$ is obviously a square integrable $(\gg,\bP)$-martingale  and from \eqref{eq:d[Z]} we see that $\langle Z\rangle^\bG_t = \int_0^t \alpha^2 (s) ds$.
	Since $\ff \subset \gg$,    $Z$ is  a square integrable $(\ff,\bP)$-martingale.
	Clearly, $Z$ has  predictable representation property in $\bF$ since $\alpha >0$ and $V$ has predictable representation property in $\bF^V$.
	Therefore for such $Z$ we see that conditions B1-B2 are satisfied.
	By definition of $\bG$ we have that $\cG_0 $ is trivial, so   B3 holds. 	Moreover, we additionally assume that, for $t\geq 0$,
	\begin{equation}\label{eq:intc2}
	\bE \Big( \sup_{s \leq t} \Big( \frac{|\gamma_s| }{\alpha(s)} i(s) + (1 - i(s)) \Big|\frac{\kappa_s}{\phi(s)}\Big|\Big) \Big)
	< \infty.
	\end{equation}
	From \eqref{e5}  and  \eqref{eq:BGMG} we have
	\begin{equation*}
	H_t := \frac{d \langle \wh{M}^\bG, Z \rangle_t }{ d \langle Z \rangle_t } = \frac{\gamma_t i(t) \alpha(t) + \kappa_t \phi(t) \lambda(t) }{i^2(t) \alpha^2(t) + \phi^2(t) \lambda(t)}	=\frac{\gamma_t }{\alpha(t)} i(t) + (1 - i(t)) \frac{\kappa_t}{\phi(t)}\quad dt \otimes d \bP\ a.e.
	\end{equation*}
	Hence, assumption \eqref{eq:intc2} and Remark \ref{rem:4.1} imply that the predictable projection of the process ${H}$ exists,  so B4 holds.

	Now, using Theorem \ref{not-adapted} we obtain that  $^{o,\bF}\!X$ is a special semimartingale whose characteristics are expressed in terms of $h$ given by \eqref{eq:h}. We will now proceed with computation of $h$.
	Since $d\langle Z\rangle_t= \alpha^2(t) dt$, $\alpha > 0$,  we see from \eqref{eq:h} that $h_t={^{p,\bF}\!H_t}$ for $t\geq 0$ outside of an evanescent set, so that $h_t = \bE( H_t | \cF_{t-})$ for $t>0$ outside of an evanescent set. Thus we may write
	\[
	h_t = \frac{\bE( \gamma_t | \cF_{t-}) }{\alpha(t)} i(t) + (1 - i(t)) \frac{\bE( \kappa_t | \cF_{t-})}{\phi(t)} 
	.
	\]
	In view of Theorem \ref{not-adapted} again,
	having the above form of $h$, we find the $\bF$-characteristics of $^{o,\bF}\!X$. Since  $d\langle Z^c\rangle_t= i^2(t) \alpha^2(t) dt$, the $\ff$-characteristics of $^{o,\bF}\!X$ are
	\begin{equation*}
	C^\bF = \int_0^\cdot h^2_s i^2(s) \alpha^2(s)ds = \int_0^\cdot (\bE( \gamma_s | \cF_{s-}))^2 i(s)ds,\ t\geq 0,
	\end{equation*}
	and, for any $ A \in \mathcal{B} (\bR)$,
	\begin{align*}
	\nu^\bF(A,dt) &= \Big(\int_\bR \1_{A \setminus \set{0}}(h_t x)\delta_{\phi(t)}(dx)\Big)\lambda(t)dt
	\\
	&=\Big(\int_\bR \1_{A \setminus \set{0}}(\bE ( H_t | \cF_{t-})  x)\delta_{\phi(t)}(dx)\Big)(1 - i(t)) \lambda(t)dt
	\\
	&=\Big(\int_\bR \1_{A \setminus \set{0}}\Big((1 - i(t)) \frac{\bE( \kappa_t | \cF_{t-})}{\phi(t)}  x\Big)\delta_{\phi(t)}(dx)\Big)(1 - i(t)) \lambda(t)dt
	\\
	&= \Big(\int_\bR \1_{A \setminus \set{0}}(x)\delta_{ \bE( \kappa_t | \cF_{t-}) }(dx)\Big)\lambda(t)dt
	,\ t\geq 0,
	\end{align*}	
	and finally the first $\bF$-characteristic is given by
	\begin{equation*}
	B^\bF_t = \int_0^t \left({^{o,\bF}}(b_s +\1_{|\kappa_s| > 1}\kappa_s \lambda(s))  -   \1_{|\bE( \kappa_s | \cF_{s-})| > 1} \bE( \kappa_s | \cF_{s-}) \lambda(s) \right )ds, \ t\geq 0.
	\end{equation*}

	\noindent b)
	Now, we take $\bF = \bF^M\subseteq \bG$ and  $Z=M$.	 We additionally assume that
	\begin{equation}\label{eq:intc5}
	\bE \Big( \sup_{s \leq t}  |\kappa_s|
	\Big)
	< \infty,\ t\geq 0.
	\end{equation} 	Then, proceeding in a way analogous to what is done in a) above,  we compute
	\[
	d \langle \wh{M}^\bG, Z \rangle_t  = d \Big\langle \int_0^\cdot \gamma_s d W_s + \int_0^\cdot \kappa_s d M_s, M \Big\rangle_t\]
	\[=
	\gamma_t  d \langle W, M \rangle_t + \kappa_s d \langle M \rangle_t =  \kappa_t d \langle M \rangle_t=  \kappa_t d \langle Z \rangle_t,\ t\geq 0,
	\]
	so that
	\begin{equation*}
	H_t= \frac{d \langle \wh{M}^\bG, Z \rangle_t }{ d \langle Z \rangle_t } = \kappa_t \quad dt \otimes d \bP\ a.e.\
	\end{equation*}
	Hence, assumption \eqref{eq:intc5} implies that the predictable projection of the process ${H}$ exist and we conclude that B4 holds.
	Then,  from Theorem \ref{not-adapted} and the fact that $Z^c = M^c = 0$ we obtain that $\bF$-characteristics of $^{o,\bF}\!X$ are given by 	\[
	B^\bF_t = \int_0^t \big( \ ^{o,\bF} ( \beta_s + \kappa_s\1_{|\kappa_s| > 1 }\lambda(s) )  -  \1_{|\bE( \kappa_s | \cF_{s-})| > 1} \bE( \kappa_s | \cF_{s-}) \lambda(s) \big) ds,  \quad C^{\bF}_{t} = 0,
	\]
	\[
	\nu^{\bF} (dx,dt )  = \delta_{ \bE(\kappa_t| \cF_{t-}) }(dx)  \1_{\set{ \bE(\kappa_t| \cF_{-t}) \neq 0}} \lambda(t) dt, \ t\geq 0.
	\]
	
	\noindent c) Here, we take $\bF = \bF^W\subseteq \bG$ and  $Z=W$. We additionally assume that
	\begin{equation}\label{eq:intc6}
	\bE \Big( \sup_{s \leq t} |\gamma_s|
	\Big)
	< \infty,\ t\geq 0.
	\end{equation}	
	
	We have
	\[
	d \langle \wh{M}^\bG, Z \rangle_t  = d \Big\langle \int_0^\cdot \gamma_s d W_s + \int_0^\cdot \kappa_s d M_s, W \Big\rangle_t\]
	\[=
	\gamma_t d \langle W \rangle_t + \kappa_t d \langle M,W \rangle_t = \gamma_t d \langle W \rangle_t, \ t\geq 0.
	\]
	Thus
	\begin{equation*}
	\frac{d \langle \wh{M}^\bG, Z \rangle_t }{ d \langle Z \rangle_t } = \gamma_t \quad dt \otimes d \bP\ a.e.\
	\end{equation*}
	Hence, assumption \eqref{eq:intc6} implies that the predictable projection of the process ${H}$ exist and thus assumption B4 holds.
	Then, applying Theorem \ref{not-adapted}, we conclude  that $\bF$-characteristics of $^{o,\bF}\!X$ are given by
	\[
	B^\bF_t = \int_0^t \ ^{o,\bF} ( \beta_s + \kappa_s\1_{|\kappa_s| > 1 } \lambda(s)) ds,  \quad C^{\bF}_{t} = \int_0^t (\bE( \gamma_s | \cF_{s-}) )^2 ds,\ t\geq 0,\  \quad \nu^{\bF} \equiv 0,
	\]
	where the third equality follows from $ \nu^{Z,\bF}  = \nu^{W,\bF} \equiv 0$.
	
	\noindent d) Let us take $\bF = \bF^W\subseteq \bG$ and  $Z=W$. Moreover, take $\lambda(s) = \lambda>0$, $\beta_s=\lambda$, $\kappa_s=1$ and $\gamma_s=0$ for $s\geq 0$, and $X_0=0$.  Thus $X = N$ is a Poisson process with intensity $\lambda$.
	Clearly $X$ is a special $\gg$-semimartingale with $\bG$-characteristics $(B^\bG,C^\bG,\nu^\bG)$, where
	\[
	B^{\bG}_t = \lambda t , \quad C^\bG_t=0,  \quad \nu^\bG (dx,dt) = \lambda \delta_1(dx) dt,
	\]
	Applying  Theorem \ref{not-adapted}  we see that $\bF$-characteristics of $^{o,\bF}\!X$ are given by
	\[
	B^\bF_t = \int_0^t \ ^{o,\bF} \lambda ds = \lambda t,  \quad C^{\bF}_{t} = 0, \quad \nu^{\bF} \equiv 0.
	\]	
	So a purely discontinuous special semimartingale  $X$ admits continuous optional projection $^{o,\bF}\!X$ .
	\qed
\end{example}

We will now present an  example where $X$ is a continuous special $\bG$--semimartingale, and $^{o,\bF}X$ is a purely discontinuous special $\bF$--semimartingale.

\begin{example} \label{ex:5.7}
	Consider a standard Brownian motion $W$. Let $X=W$ and take $\bG=\bF^X$. The $\bG$-characteristics triple of $W$ is $(0,C^\bG,0)$,  where $C^\bG_t=t$. In particular, we have $b^\bG=0$ and $a^\bG=1$.
	Next, define the filtration $\bF $ as
	\[
	\cF_t=\cF^X_n, \quad t\in [n,n+1),\ n=0,1,2,\ldots
	\]
	The optional projection of $X$ on $\bF$ exists and is given as
	\[
	^{o,\bF}X_t=X_n, \quad t\in [n,n+1),\ n=0,1,2,\ldots
	\]
	In order to compute the  $\bF$-characteristics of $^{o,\bF}\!X$ we first observe that
	the canonical semimartingale representation of $^{o,\bF}X$, with respect to the standard truncation function, is given as
	\begin{equation}\label{eq:oFX}
	^{o,\bF}X=x \ast \mu=(x\1_{|x|\leq 1}) \ast \nu+(x\1_{|x|\leq 1}) \ast (\mu-\nu)+(x\1_{|x|> 1}) \ast \mu,
	\end{equation}
	where
	\[\mu(dt,dx)=\sum_{n\geq 1}\delta_{(n,X_n-X_{n-1})}(dt,dx),\]
	and
	\[
	\nu(\omega,dt,dx)=\sum_{n\geq 1}\, \delta_{n}(dt) \frac{1}{\sqrt{2\pi}}e^{-\frac{x^2}{2}}dx.
	\]
	From \eqref{eq:oFX} we obtain
	\begin{equation*}
	B^\bF_t=\int_0^t\int_{ |x|\leq 1}x\nu(ds,dx)
	=
	\int_{(0,t] \times [-1,1]}   x \sum_{ n \geq 1} \delta_n(ds) \left(\frac{1}{\sqrt{2\pi}}e^{-\frac{x^2}{2}}dx \right)
	\end{equation*}
	\[
	=\int_{(0,t] }   \sum_{ n \geq 1}  \Big(\int_{-1}^1 x \frac{1}{\sqrt{2\pi}}e^{-\frac{x^2}{2}}dx \Big) \delta_n(ds)
	=0.
	\]
	
	Thus, the  $\bF$-characteristics triple of $^{o,\bF}X$ is $(0,0,\nu^\bF)$, where
	\[
	\nu^\bF=\nu.
	\]
	\qed
\end{example}

\begin{example}\label{ex:MonMod}
	Let us consider the case where $\bF$ is a Brownian filtration, $\bG$ its progressive enlargement with a strictly positive random time $\tau$. Taking $X_t=\ind_{\{\tau \leq t\}}$, $t\geq 0$ we have (cf. \cite{AksJea2017}),
	\[
	\bG=\bF \triangledown \bF^X,
	\]
	where $\bF \triangledown \bF^X$ is the smallest right-continuous filtration which contains $\bF$ and $\bF^X$.
	Now, we define the Az\'ema supermartingale $A$ by
	\[A_t=\bP(\tau>t \vert \F_t),\quad t\geq 0,\]
	and we write its Doob-Meyer decomposition as
	\[A_t=m_t - b_t,\quad t\geq 0,\]
	where $m$ is an $\bF$-martingale, and $b$ is an $\bF$-predictable, increasing process which is  the $\bF$-dual predictable projection  of $X$.
	We assume that  $\tau$ satisfies the following Jacod absolute continuity assumption
	\begin{equation}\label{jacod}\bP(\tau >s \vert \F_t)=\int_s^\infty \alpha _t(u)du,\quad s,t\geq 0,\end{equation}
	where, for any $u\geq 0$, the process $\alpha_\cdot(u)$ is a positive continuous $\bF$-martingale and the map  $(\omega, t,u) \rightarrow \alpha_t(\omega;u)$ is $\tilde{\cP}_\bF$-measurable.
	Using the fact that $\int_0^\infty \alpha _t(u)du=\bP(\tau >0)=1$ and $\alpha_\cdot(u)$ is a martingale, it is shown in Proposition 4.1  in \cite{ElkJeaJia2010} that
	\[ db_t   = \alpha _t (t)dt\]
	and
	\begin{equation}\label{m}m_t = \E\Big(\int_0^\infty \alpha _u(u)du\vert \F_t \Big)= 1+ \int_0^t \alpha _u(u)du-\int _0^t \alpha_t(u)du.\end{equation}
	Note that in the above set-up, the  process $A$ is continuous.
	
	The process $X$ is a special $\bG$-semimartingale and we know (cf. Corollary 5.27 in \cite{AksJea2017}) that its canonical decomposition is given as
	\[X= M^\bG+B^\bG,\]
	and its $\bG$-characteristics are $(B^\bG,0,\nu^\bG)$, where
	\[B^\bG_t=\int_0^t {(1 - X_{s})} \frac{db_s}{A_{s}}=\int_0^t \frac{{ (1 - X_{s})} \alpha_s(s)}{A_{s}}ds,\quad t\geq 0\]
	and
	\[\nu^\bG(dt,dx)=\delta_{1}(dx) \frac{{(1 - X_{t-})} \alpha_t(t)}{A_{t}}dt.\]
	In particular, note that here we have $b^\bG_t=\frac{{ (1 - X_{t})} \alpha_t(t)}{A_{t}}$ and $a^\bG_t=1$.
	
	Now, using Lemma \ref{lem:heroic-result} and observing that $^{o,\bF}X=1-A$  we can easily compute the first $\bF$-characteristic of $^{o,\bF}X$,
	\[B^\bF_t=\int_0^t {{^{o,\bF}}\left (\frac{{(1-X_{s})} \alpha_s(s)}{A_{s}}\right )}ds=\int_0^t \alpha_s(s)ds.\]
	Next, recalling that $A$ is a continuous process  we  conclude that $\nu^\bF=0$. Moreover, we see that $C^\bF = \langle m \rangle$.
	This completes the computation of the $\bF$-characteristics of $^{o,\bF}X$ which are $(B^\bF, \langle m \rangle, 0)$. \qed
\end{example}

The next example is, in a sense, opposite to Example \ref{ex:5.7}: here, $X$ is a purely discontinuous special $\bG$--semimartingale, and  $^{o,\bF}X$ is a continuous special $\bF$--semimartingale. Thus, this example complements Example \ref{ex:Privault} d).

\begin{example}
	Let $\bF$ be a Brownian filtration and $\bG$ its progressive enlargement with a strictly positive random time $\tau \in \F_\infty$ satisfying Jacod's absolute continuity assumption \eqref{jacod} with some density $\alpha _t(u),\ t,u\geq 0$.
	Such a random time can be defined as $\tau:= \psi (\int_0^\infty f(t)dW_t)$, where $\psi$ is a differentiable, positive and strictly increasing function, and $W$ is a real valued standard $\bF$-Brownian motion (see  \cite{ElkJeaJiaZar2014}).
	Let $\widehat X$ be the compensated martingale
	\[
	\widehat X_t=\ind_{\{\tau\leq t\}}-\int_0^{t\wedge \tau} \frac{\alpha_s(s)}{A_s}ds, \quad t\geq 0.
	\]
	We see that its $\bG$-characteristic triple is $(0,0,\nu^\bG)$ where, as in the previous example,
	\[\nu^\bG(dt,dx)=\delta_{1}(dx) \frac{\ind_{\{t<\tau\}} \alpha_t(t)}{A_{t}}dt.\]
	The $\bF$-optional projection of $\widehat X$, say $\upsilon$, is a continuous martingale, which is not constant. Indeed, note that if $\upsilon$ were constant then $\upsilon_\infty=\upsilon_0=0$. Given that, one has $\widehat X_ \infty=1-\int_0^{ \tau} \frac{\alpha_s(s)}{A_s}ds\in \F_\infty$ and $\upsilon_\infty=\widehat X_ \infty$. But since $\upsilon_\infty=0$, then $\widehat X_\infty=0$, and $\widehat X$ being a martingale would be null, which it is not. This is a contradiction, showing that $\upsilon$ is not constant. Consequently, its $\bF$ characteristic triple is $(0,C^\bF,0)$, with $C^\bF\ne 0$. \qed
\end{example}

\section{Conclusion and open problems for future research}\label{sec:conclusion}

As stated in the Introduction this paper is meant to initiate a systematic study of the change of properties of semimartingales  under shrinkage of filtrations and, when appropriate,  under respective projections.

Given its pioneering nature the study originated here leads to numerous open problems and calls for extensions in numerous directions.  Below, we indicate some such open problems and suggestions for continuation of the research presented in this paper.

The results presented in this paper use several non-trivial assumptions. A natural direction for continuation of the present work will be to try to eliminate some of these assumptions.

Recall the decomposition \eqref{eq:important-decompostion}
\begin{align*}\nonumber
^{o,\bF} X_t &= \ ^{o,\bF}\!X_0 + ^{o,\bF}\!\wh{M}^\bG_t+^{o,\bF}\! \wh{B}^\bG_t
\\ \label{eq:important-decompostion}
&=  X_0 + ^{o,\bF}\!\wh{M}^\bG_t+^{o,\bF}\! \wh{B}^\bG_t  - \int_0^t\,  {^{o,\bF}}\! (\wh{b}^\bG a^\bG)_udu +  \int_0^t\,  {^{o,\bF}} (\wh{b}^\bG a^\bG)_udu.
\end{align*}
As it was shown in the proof of Theorem \ref{not-adapted}, if the immersion hypothesis B2 is postulated, then the martingale  $\wh{M}^B_t={^{o,\bF} \wh{B}^\bG_t}- \int_0^t\,  {^{o,\bF}} (\wh{b}^\bG a^\bG)_udu$ is null.
Therefore it does not intervene in the representation of the $\bF$-characteristics of $^{o,\bF}\!X$. If however the martingale $\wh{M}^B$ is not null, then the computation of the $\bF$-characteristics of $^{o,\bF}\!X$ in terms of the $\bG$-characteristics of $X$ is much more challenging, and perhaps may not be doable.

The immersion hypothesis B2 postulated Theorem \ref{not-adapted} is also heavily exploited in computation of the second $\bF$-characteristic of $X$, that is in computation of $C^\bF$. In fact, computation  of $C^\bF$ in terms of $\bG$-canonical decomposition appears to be much more difficult, or even impossible, without the hypothesis B2, as the following reasoning shows: Assume that $\bF$ is a Brownian filtration generated by $W$, so that $W$ enjoys the predictable representation property in $\ff$. Also, take   $\bG$  to be the progressive enlargement of $\bF$ by a random time $\tau$.\footnote{See e.g. Chapter 5 in \cite{AksJea2017} for the concept of the progressive enlargement of filtrations.} Assume that there exists  a $\bG$-predictable integrable process $\mu$ such that $W^\bG$ defined  for any $t\in \r_+$   as  $$W^\bG_t = W_t+\int_0^t \mu_sds$$ is a $\bG$-martingale (hence, a $\bG$-Brownian motion).
Then, any  $\bG$-martingale $X$ can be written as
\[X_t =X_0+\int_0^t \psi_sdW^\bG_s+ M ^ \bot_t,\quad t\geq 0,\]
where $\psi$ is a $\bG$-predictable process and $M^\bot$ a $\bG$-martingale orthogonal to $W^\bG$ (in fact, it is a purely discontinuous martingale).  Moreover,  one can show  (using the same methodology as in  \cite{gapeev:hal-02120343}) that $^{o,\bF}X$, which is an $\ff$-martingale, has the form
\[^{o,\bF}X_t= {^{o,\bF}X_0}+\int_0^t \gamma_sdW_s,\quad t\geq 0,\]
where $\gamma$ satisfies $\gamma_t =\bE( \psi_t +\mu_t X_t \vert \F_t)$. So here we have that
\[C^\bF_t=\int_0^t (\bE( \psi_s +\mu_s X_s \vert \F_s))^2ds,\]
\[C^\bG_t=\int_0^t \psi^2_s ds.\]
Clearly, $C^\bG$ alone does not suffice to compute $C^\bF$, unless $\mu \equiv 0$ -- i.e., $\bF$ is immersed in $\bG$.   In fact, it is not clear at all, how to compute the $C^\bF$ characteristic of $X$ in terms of the canonical decomposition and $\bG$ characteristics of $X$.

The discussion above points to an important open problem: extend, if possible, the results of Theorem \ref{not-adapted} to the case when the immersion hypothesis B2 is abandoned, and extend the result of \cite{gapeev:hal-02120343} to the case of general continuous semi-martingales.

Another challenging problem for future research is weakening of the predictable representation property condition B1, and replacing it with the postulate of the weak predictable representation property condition for $Z$, that is with the postulate that
every local $\ff$--martingale $Y$ admits the representation

\[Y_t=Y_0+\psi\cdot Z^c_t + \xi \ast \widetilde \mu^Z_t,\quad t\geq 0,\]
where $\psi$ is an $\ff$-predictable process, $\xi$ is a $\tilde{\cP}_\bF$-measurable function, $Z^c$ is the continuous martingale part of $Z$, and $\widetilde \mu^Z$ is the $\bF$--compensated  measure of jumps of $Z$.


Finally, it might be worthwhile to study the following interesting question: 
Suppose we have two different semimartingales $X$ and $Y$ with different laws but with the same characteristics in $\bG$. Will their characteristics in $\bF$ be the same as well?

\bibliographystyle{alpha}
\bibliography{Math_Fin}

\end{document}